\documentclass[12pt,letterpaper,reqno]{amsart}

\usepackage{amsthm}
\usepackage{mathrsfs}
\usepackage{amsfonts,amssymb,amsmath}
\input amssym.def 
\input amssym.tex

\usepackage[dvips]{graphicx}
\usepackage{hyperref}
\usepackage{thmtools}

\usepackage{enumerate}

\declaretheoremstyle[
bodyfont=\normalfont,
]{remstyle}

\declaretheorem[style=remstyle, name=Remark,numberwithin=section]{rem}

\newtheorem{defi}{\bf Definition}

\newtheorem{lemma}{\bf Lemma}

\newtheorem{theorem}{\bf Theorem}
\newtheorem{claim}{\bf Claim}[section]

\newtheorem{corollary}{\bf Corollary}[section]

\newcommand\be{\begin{eqnarray*}}
\newcommand\ee{\end{eqnarray*}}
\newcommand\beq{\begin{equation}}
\newcommand\eeq{\end{equation}}

\newcommand\ben{\begin{eqnarray}}
\newcommand\een{\end{eqnarray}}

\newcommand{\norm}[1]{\left\lVert#1\right\rVert}

\usepackage{xcolor}
\newcommand*{\missingreference}{\colorbox{red}{?reference?}}
\newcommand*{\missingcitation}{\colorbox{red}{?citation?}}
\makeatletter
\def\@setref#1#2#3{%
  \ifx#1\relax
   \protect\G@refundefinedtrue
   \nfss@text{\reset@font\missingreference}%
   \@latex@warning{Reference `#3' on page \thepage \space
             undefined}%
  \else
   \expandafter#2#1\null
  \fi}
\def\@citex[#1]#2{\leavevmode
  \let\@citea\@empty
  \@cite{\@for\@citeb:=#2\do
    {\@citea\def\@citea{,\penalty\@m\ }%
     \edef\@citeb{\expandafter\@firstofone\@citeb\@empty}%
     \if@filesw\immediate\write\@auxout{\string\citation{\@citeb}}\fi
     \@ifundefined{b@\@citeb}{\hbox{\reset@font\missingcitation}%
       \G@refundefinedtrue
       \@latex@warning
         {Citation `\@citeb' on page \thepage \space undefined}}%
       {\@cite@ofmt{\csname b@\@citeb\endcsname}}}}{#1}}
\makeatother

\usepackage[final]{showkeys} 
\usepackage{cleveref}

\begin{document}

\date{\today}

\title[Bourgain-Chang proof of the weak Erd\H{o}s-Szemer\'edi conjecture] 
{Bourgain-Chang proof of the weak Erd\H{o}s-Szemer\'edi conjecture}

\author{Dmitrii Zhelezov}
\thanks{Alfr\'{e}d R\'{e}nyi Institute of Mathematics} 
\address{Alfr\'{e}d R\'{e}nyi Institute of Mathematics, 
Re\'altanoda utca 13-15., 
1053 Budapest, Hungary} \email{dzhelezov@gmail.com}

\date{\today}

\begin{abstract}
	This is an exposition of the following `weak' Erd\H{o}s-Szemer\'edi conjecture for integer sets proved by Bourgain and Chang in 2004. For any $\gamma > 0$ there exists $\Lambda(\gamma) > 0$ such that for an arbitrary $A \subset \mathbb{N}$, if $|AA| \leq K|A|$ then
	$$
		E_{+}(A) \leq K^{\Lambda}|A|^{2+\gamma}.
	$$
\end{abstract}

\maketitle

\section*{Notation}
The following notation is used throughout the paper. The expressions $X \gg Y$, $Y \ll X$, $Y = O(X)$, $X = \Omega(Y)$ all have the same meaning that there is an absolute constant $c$ such that $|Y| \leq c|X|$. For a graph $G$, $E(G)$ denotes the set of edges and $V(G)$ denotes the set of vertices. If $X$ is a set then $|X|$ denotes its cardinality. 

We write $A \approx B$ if $A \leq B \leq 2A$.

For sets of numbers $A$ and $B$ the sumset $A + B$ is the set of all pairwise sums  $\{ a + b: a \in A, b \in B \}$, and similarly $AB$, $A-B$ denotes the set of products and differences, respectively. 

For a number $x$ and a set $Y$ the expression $xY$ denotes the set $\{ xy : y \in Y \}$, and similarly for the additive shift $x + Y ;= \{ x + y: y \in Y  \}$.

If $G \subset A \times B$ is some graph  then $A \stackrel{G}{+} B$ denotes the the restricted sumset $\{ a + b : (a, b) \in G \}$. 

%We write $A^{(k)}$ and $kA$ for the iterated product set $A \ldots A$ and the sumset $A + \ldots + A$ respectively ($A$ is repeated $k$ times).

For a vertex $v$ of $G$ we write $N_G(v)$ for the set of neighbors of $v$ in $G$. The subindex $G$ may be omitted if it is clear to which graph the vertex belongs.

The additive energy $E_+(A, B)$ is defined as the number of additive quadruples $(a_1, b_1, a_2, b_2)$ such that 
$$
a_1 + b_1 = a_2 + b_2.
$$
We write $E_+(A)$ or simply $E(A)$ for $E_+(A, A)$. 

For a set $A$ we write $1_A$ for the indicator function of $A$ and the convolution $f * g$ is defined as
$$
f * g(x) := \sum_y f(y) g(x-y),
$$
where the sum is taken over the support of $f$. In particular, one verifies that
$$
E_+(A, B) = \sum_x  (1_A* 1_B)^2(x) = \norm{1_A* 1_B}^2_2.
$$

\section{Introduction}

	In the present exposition we give a slightly simplified proof of the result due to Bourgain and Chang \cite{BourgainChang} below. 

\begin{theorem} \label{thm:maintheorem}
For any $\gamma > 0$ there exists $\Lambda(\gamma) > 0$ such that for an arbitrary $A \subset \mathbb{Z}$ if $|AA| \leq K|A|$ then
	$$
		E_{+}(A) \leq K^{\Lambda}|A|^{2+\gamma}.
	$$
\end{theorem}
Bourgain and Chang actually proved a much stronger result which is as follows.
\begin{theorem} \label{thm:BCstrong}
Given $\gamma > 0$ and $q > 2$, there is a constant $\Lambda(\gamma, q)$ such that if $A \subset \mathbb{Z}$ is a finite set with  $|A| = N; |A A| < KN$, then
$$
\norm{ \sum_{n \in A} c_n e(n \theta)}_{L_q(\mathbb{R}/\mathbb{Z})} \leq K^\Lambda N^\gamma \left( \sum_{n \in A} c^2_n\right)^{1/2} 
$$
with the usual notation $e(x) := \exp(2 \pi i x)$.
\end{theorem} 

In particular,  taking $c_n := 1$ in Theorem \ref{thm:BCstrong} and expanding the $L_q$ with $q = 2k$, it follows that for an arbitrary $\gamma > 0$ and integer $k \geq 2$ there exists $\Lambda(\gamma, k)$ such that
$$
	E_k(A) \leq K^{\Lambda(k, \gamma)}|A|^{k+\gamma},
$$
where $E_k$ is the $k$\emph{-energy}, defined as
$$
E_k(A) := |(a_1, \ldots, a_k, a'_1, \ldots, a'_k) : a_1 + \ldots + a_k = a'_1 + \ldots + a'_k|. 
$$

Finally, the pinnacle of \cite{BourgainChang} is the following infinite growth sum-product theorem.
\begin{theorem} \label{thm:BCexpansion}
For any  $b > 0$ there exists integer $k(b)$ with the property that
$$
|A+ \ldots + A| + |A \ldots A| \geq |A|^b
$$ 
for any integer set $A$ ($A$ is taken $k$ times in the iterated sumset and the product set above). 
\end{theorem}
This impressive result can be quickly deduced from Theorem \ref{thm:BCstrong}, see Proposition 2 in \cite{BourgainChang}.

	The main purpose of this exposition is to present the combinatorial arguments of \cite{BourgainChang} in a way which is more familiar for mathematicians working in arithmetic combinatorics, in particular on problems related to the sum-product phenomena. The original paper exploits the machinery of trigonometric polynomials which makes it somewhat hard to absorb for readers with little background in harmonic analysis. Since the weak Erd\H{o}s-Szemer\'edi conjecture for real sets is still wide open, we hope that a better understanding of the Bourgain-Chang method may help to make progress on the real case or at least highlight the obstacles.

	We are going to prove only Theorem \ref{thm:maintheorem} which is admittedly weaker then Theorem 2 or Theorem \ref{thm:BCexpansion}. Nevertheless, we have decided to pay such a price in order to streamline the exposition. While the machinery of trigonometric polynomials and $\Lambda_q$-constants turns out to be more robust, we believe that the proof of Theorem \ref{thm:maintheorem} presented below already contains all essential ingredients needed for the general case. The case of $E_+$, however, allows one to make all the arguments on the `physical side' (basically, because the $L_2$ norm is invariant under the Fourier transform) which makes the proof purely combinatorial and elementary.
	
	We therefore suggest using the current note as a warm-up or as a supplementary reading for \cite{BourgainChang}. We will occasionally skip some intermediate steps in the calculations which we think are routine, referring the reader to the original paper for details. Again, our motivation here is to give a somewhat informal sketch of the arguments and convince the reader that the whole setup \emph{must} work. The details may then be filled by reading  \cite{BourgainChang}, which is an impeccable and rigorous piece. Any errors, gaps, inconsistencies or sloppy explanations are solely due to the author of this note.

\section{Motivating examples}

We start with some motivating examples. Let $Y$ be an integer set which can be decomposed as a disjoint union
$$
		Y = \bigcup_{i \in I} a_i X_i
$$
with $I = {1, \ldots, N}$, some distinct numbers $a_i \in \mathbb{Z}$ and integer sets $X_i$. Applying the Cauchy-Schwarz inequality twice, one can then write for the energy (all the summations are through the indices in $I$)
\ben
E^{1/2}_+(Y) &=& \norm{ \sum_{i, j} 1_{a_iX_i} *1_{a_iX_i}}_2 \leq \sum_{i,j} \norm{1_{a_iX_i} *1_{a_iX_i}}_2 \\
					 &=&  \sum_{i, j} E^{1/2}(a_iX_i, a_jX_j) \leq \sum_{i, j}E^{1/4}(X_i) E^{1/4}(X_j) \label{eq2} \\
					 &=& (\sum_i E^{1/4}(X_i))^2 \leq N \sum_i E^{1/2}(X_i).				\label{eq3} 
\een

We can thus bound the energy $E(Y)$ by the energy of the constituents. Of course, without any prior knowledge about $X_i$'s and $a_i$'s the inequality above doesn't give much, but one might hope that under certain conditions on $a_i, i \in I$ there is almost no additive interaction between the sets $a_iX_i$ and $a_jX_j$ in (\ref{eq2}) and a better bound
\beq \label{eq:setup}
E^{1/2}_+(\bigcup_{i \in I} a_i X_i) \leq \psi(|I|)\sum_i E^{1/2}(X_i)
\eeq 
holds with some sublinear function $\psi$. We make the following definition.

\begin{defi}[Separating sets]
A set $A \subset \mathbb{Z}$ is  \emph{$\psi$-separating} if the bound
\beq \label{eq:psibound}
E^{1/2}(\bigcup_{a \in A} \{ a X_a \})  \leq \psi \sum_{a \in A} E^{1/2}(X_a)
\eeq
holds for any collection of integer sets $X_a$ such that $(a, X_{a'}) = 1$ for any $a, a' \in A$.
\end{defi}

\begin{rem}
The definition of separating sets seems to be related to the notion of \emph{decoupling} in harmonic analysis extensively studied by Bourgain himself and coauthors in later works, see e.g. \cite{BourgainDemeterGuth} and references therein. Indeed, for two sets $X$ and $Y$ the additive energy $E(X, Y)$ is equal to $\langle (\hat{1_X})^2, (\hat{1_Y})^2 \rangle$, so Theorem \ref{thm:BCstrong} can be seen as a decoupling result. In particular, (\ref{eq:oneprimeseparation}) below is a manifestation of the fact that a family of \emph{bi-orthogonal} functions exhibit $L_2$-decoupling, see Terry Tao's blog \cite{TaoBlog} on decoupling for more details.
We don't know if there is a deeper connection between geometric decoupling and the sum-product phenomenon.

\end{rem}

We record (\ref{eq3}) for future use.
\begin{claim}[Trivial separation bound] \label{claim:trivialseparation}
  Any integer set $A$ is $|A|$-separating. 
\end{claim}

One may wonder if there exist sets with a good separation factor $\psi$. Here is the first motivating example when this is indeed the case. Let $p$ be a prime and $a_i = p^{k_i}$ for some $k_i \geq 0$ and 
$$
A = \bigcup_{1 \leq i \leq N} \{ a_i \}.
$$
Assume 
$$
Y = \bigcup_{1 \leq i \leq N} a_i X_i
$$
for some integer sets $X_i$ with $(p, X_i) = 1$. If now 
$(y_1, y_2, y_3, y_4) \in Y^4$ with
\beq \label{eq:energyquadruble}
	y_1 + y_2 = y_3 + y_4
\eeq 
both sides must divide the same power of $p$, so at least two of the elements $y_1, y_2, y_3, y_4$ must be in the same slice $a_iX_i$ for some $i$.  There are six possible cases for which two of $y_1, y_2, y_3, y_4$ belong to $a_iX_i$. Writing $r_{a_iX_i - a_iX_i}(x)$ (resp. $r_{Y - Y}(x)$) for the the number of ways to represent $x$ as a difference of two elements in $a_iX_i$ (resp. $Y$), we can bound the number of quadruples (\ref{eq:energyquadruble}) for given $i$ as either
$$
\sum_x r_{a_iX_i - a_iX_i}(x) r_{Y - Y} (x)
$$
or
$$
\sum_x r_{a_iX_i + a_iX_i}(x) r_{Y + Y} (x)
$$
(with the obvious modification of notation) depending on the case. Summing up, we have ($\pm$ means either plus or minus depending on the case)
\ben
E(Y) &\leq&  \sum_i \sum_x r_{a_iX_i \pm a_iX_i}(x) r_{Y \pm Y} (x) \label{eq:changfirststep} \\
		&\leq&  \sum_i \left(\sum_x r^2_{a_iX_i \pm a_iX_i}(x)\right)^{1/2} \left(\sum_x r^2_{Y \pm Y}(x)\right)^{1/2} \nonumber \\
		&=& 6 E^{1/2}(Y) \sum_i E^{1/2}(X_i) \nonumber,
\een
so
\beq \label{eq:oneprimeseparation}
E^{1/2}(Y) \leq 6 \sum_i E^{1/2}(X_i),
\eeq
which means that $A$ is $6$-separating. 

%Further, one can iterate the inequality above so that if $a_i = \prod_{1 \leq j \leq k} p^{\alpha_{i,j}}_j$, that is, all prime factors of $a_i$ are in $\{p_1, \ldots, p_k\}$ and $(p_1\ldots p_k, %X_i) =1$ then
%\beq \label{eq:primes}
%E^{1/2}(Y) \leq 6^k \sum_i E^{1/2}(X_i).
%\eeq

In what follows it will be convenient to use the \emph{prime valuation}  map which is defined as follows. Let $A$ be a rational set with the elements (after all possible cancellation in numerators and denominators) having prime factors in the set $\{ p_i \}, i \in I$. The map $\mathcal{P}_{I} : A \to \mathbb{Z}^{I}$ maps $\prod_{i \in I}p^{\alpha_i}_i$ to $(\alpha_{1}, \ldots, \alpha_{|I|})$.  It is clear, however, that since all our sets are finite there  always exists a large enough index set $I$ such that $\mathcal{P}_{I}$ is well-defined for all sets in question and is injective. We will therefore assume that this large index set is fixed and omit the subindex $I$ in $\mathcal{P}_I$ when the actual index set is not important.

Let $A$ be a finite-dimensional vector space $V$. Recall that $\mathrm{rank}(A)$ is defined as the minimal $d$ such that $A$ is contained in an affine subspace of $V$ of dimension $d$.  Next, define \emph{multiplicative dimension} of a set $A$ simply as $\mathrm{rank}(\mathcal{P}(A))$. Of course, $\mathbb{Z}^{I}$ is not a linear space since $\mathbb{Z}$ is not a field, so one should consider $\mathcal{P}(A)$ as a set naturally embedded into an ambient linear space over $\mathbb{Q}$ (or $\mathbb{R}$, which makes no difference in our case).

Recall the following lemma due to Freiman.

\begin{theorem}[Freiman's Lemma, \cite{TaoVu} Lemma 5.13] \label{thm:FreimanLemma} Let $A$ be a finite subset of a finite-dimensional space $V$ and suppose $\mathrm{rank}(A) = m$. Then
$$
	|A+A| \geq (m+1)|A| - \frac{m(m+1)}{2}.
$$
\end{theorem}

\begin{corollary} \label{corollary:Chang}
	Let $A \subset \mathbb{Z}$. Assume $|AA| \leq K|A|$. Then $A$ is $6^K$-separating.
\end{corollary}
\begin{proof}
Observe that $\mathcal{P}(AA) = \mathcal{P}(A) + \mathcal{P}(A)$ and thus $\mathcal{P}(A)$ is contained in an affine subspace of dimension at most $K$ by Freiman's Lemma. Then, by linear algebra, there exists an index set $I$ of size at most $K$ such that the map $\mathcal{P}_I: A \to \mathbb{Z}^I$ is injective. In other words, there are at most $K$ primes $p_1, \ldots, p_{|I|}$ such that each $a \in A$ can be written as (the powers $\alpha_i$ depend on $a$)
$$
a = x_a \prod^{|I|}_{i = 1} p^{\alpha_i}_i
$$
and $x_a, a \in A$ are all distinct. If we then take an arbitrary set 
$$
Y := \bigcup_{a \in A} aY_a
$$
with $(a, Y_{a'}) = 1$ for any $a, a' \in A$, we can expand 
$$
Y = \bigcup_{a \in A}  x_aY_a \prod^{|I|}_{i = 1} p^{\alpha_i}_i.
$$ 
By construction of the map $\mathcal{P}_I$ and the condition $(a, Y_{a'}) = 1$, we conclude that $x_aY_a$ don't have prime factors among $p_1, \ldots, p_{|I|}$ and thus we can repeatedly apply (\ref{eq:oneprimeseparation}) $|I|$ times for each $p_i$. It follows that
$$
E^{1/2}(Y) \leq 6^K \sum_i E^{1/2}(x_aY_a) = 6^K \sum_i E^{1/2}(Y_a),
$$
which means that $A$ is $6^K$-separating.

\end{proof}

The argument above is due to Chang and immediately implies  the Erd\H{o}s-Szemer\'edi conjecture for small $K$.
\begin{theorem}[Chang, \cite{Chang}]
	Assume $A \subset \mathbb{Z}$ with $|AA| \leq K|A|$. There is $c > 0$ such that 
	\beq \label{eq:ChangEnergy}
		E_+(A) \leq c^K |A|^2.	
	\eeq
	In particular there is $c > 0$ such that 
	\beq \label{eq:ChangSumset}
		|A+A| \geq c^{-K}|A|^2.	
	\eeq
\end{theorem}
\begin{proof}
Take $Y_a = \{ 1 \}$ for each $a \in A$ in the argument above. The bound (\ref{eq:ChangSumset}) follows from (\ref{eq:ChangEnergy}) by Cauchy-Schwarz.
\end{proof}

Chang's theorem gives a non-trivial bound only in the regime $K \ll \log |A|$. Now we turn to the case when $K$ can be as large as some small power of $|A|$.

 Let us start with an heuristic argument which rests on somewhat unrealistic assumptions but reveals the structure of the upcoming proof. Assume that there is a way to  decompose $\mathcal{P}(A)$ into a direct sum, such that 
 \beq \label{eq:sumdecomposition}
 \mathcal{P}(A) = \mathcal{P}(A_1) \oplus  \mathcal{P}(A_2)
 \eeq 
 for some sets $A_1, A_2 \subset \mathbb{N}$. Since $\mathcal{P}(A_1)$ and $\mathcal{P}(A_2)$ are orthogonal, we then have 
 \ben
 |\mathcal{P}(A)| &=& |\mathcal{P}(A_1)| |\mathcal{P}(A_2)|  \label{eq:productdecomposition}\\
 |\mathcal{P}(A) + \mathcal{P}(A)| &=& |\mathcal{P}(A_1) + \mathcal{P}(A_1)| |\mathcal{P}(A_2)+ \mathcal{P}(A_2)|  \label{eq:doubling}
\een

If we define $K_1 := |A_1A_1|/|A_1|$ and $K_2 :=  |A_2A_2|/|A_2|$ then (\ref{eq:productdecomposition}) and (\ref{eq:doubling})  give
\beq
	K_1K_2|A| = K_1K_2 |A_1||A_2| \leq |AA| \leq K|A|  \label{eq:K_1K_2K}
\eeq
so $K_1K_2 \leq K$. 

Setting $|A| = N$, assume further that $|A_1| \approx |A_2| \approx N^{1/2}$ and, moreover, that $\mathcal{P}(A_1)$ and  $\mathcal{P}(A_2)$ can be iteratively decomposed further into direct sums in a similar way. In other words, we assume that for any $l \ll \log \log A$ there is a decomposition 
\ben
\mathcal{P}(A) = \bigoplus^{2^l}_{i = 1}  \mathcal{P}(A_i)\\ \label{eq:AdecompositionAi}
|A_i| \approx N^{1/2^l} \label{eq:Aisizes}
\een
which, iterating (\ref{eq:K_1K_2K}) and taking logarithms, gives
\beq \label{eq:logKs}
	\sum^{2^l}_{i = 1} \log K_i  \leq \log K,
\eeq
where $K_i :=  |A_iA_i|/|A_i|$. Take $l = \lfloor \log \log K \rfloor $, fix an arbitrary (large) constant $C > 0$ and let 
$$
I := \{ i : K_i > C\}.
$$ 
By (\ref{eq:logKs}) we have 
$$
|I| \leq \frac{1}{C} \log K,
$$
so the size of the set  
$$
A' := \prod_{i \in I} A_i 
$$
is at most $N^{1/C}$ by (\ref{eq:Aisizes}). We can rewrite (\ref{eq:AdecompositionAi}) as
$$
	A = \prod_{i \notin I} A_i A' = \bigcup_{\stackrel{a_i \in A_i}{i \notin I}} \left( \prod a_i \right) A',
$$
and an iterative application of Corollary \ref{corollary:Chang} for $A_i, i \notin I$ gives (sum is over all elements in $\prod_{i \notin I} A_i$)
\beq \label{eq:iterationChang}
	E^{1/2}(A) \leq  \left( \prod_{i \notin I} 6^C \right) \sum E^{1/2} (A' ) \leq K^{10C} |A'|^{3/2} \prod_{i \notin I } |A_i|  . 
\eeq
so
$$
E(A) \leq K^{20C} |A'| |A|^2 \leq K^{20C} N^{1/C} |A|^2.
$$
Taking $C > 0$ large enough we recover the claim of Theorem \ref{thm:maintheorem}.

Of course, the assumption (\ref{eq:sumdecomposition}) is too strong to be true and one can easily come up with examples of sets which cannot be decomposed into a direct sum. However, in order to iterate Corollary \ref{corollary:Chang} in (\ref{eq:iterationChang}) it suffices that the set $\mathcal{P}(A)$ "fibers" into sets with controlled separating constants, which is a much weaker assumption than (\ref{eq:sumdecomposition}). 

The claim below illustrates this observation.

%. It turns out that the separation constants can be nested under a much weaker

% In what follows we switch to the additive notation by assuming that all the sets have prime factors in a finite set of primes $\mathcal{P}_0$ so after the prime power projection %$\mathcal{P}$ we are left with additive sets lying in $\mathcal{R}_0 := \prod_{p \in \mathcal{P}_0}\mathbb{Z}_{\geq 0}$. In particular, we have $|\mathcal{P}(A) + \mathcal{P}(A)| \leq K|%A|$.
 
% The claim below is motivated by the structure of the proof which is an iterative decompostion of the original set $A$. 

\begin{claim} \label{claimPsi}
Assume that a set $A$ decomposes as
\ben  
A = \bigcup_{b_i \in B} b_i C_i,  \label{eq:Adecomposition}
\een so that $(b_i, c_j) = 1$ for any $b_i \in B, c_j \in C_j$ (this means that $\mathcal{P}(b_i)$ and $\mathcal{P}(C_j)$ are orthogonal). 
Assume also that $B$ is $\psi_1$-separating and $C_i$ is $\psi_2$-separating for each $i$. Then $A$ is $\psi_1\psi_2$-separating.
\end{claim}

\begin{proof}
Let
$$
	Y = \bigcup_{a \in A} aX_a. 
$$
where $X_a$  are some sets with $(a, X_{a'}) = 1$ for  $a, a' \in A$.
Then we can write
$$
Y = \bigcup_{b_i \in B} b_i \bigcup_{c_j \in C_i} c_jX_{i, j}
$$
Since $B$ is $\psi_1$-separating and each $C_i$ is  $\psi_2$-separating we have
$$
E^{1/2}(Y) \leq \psi_1 \sum_{b_1 \in B} E^{1/2}(\bigcup_{c_j \in C_i} c_jX_{i, j}) \leq \psi_1 \psi_2 \sum_{i, j} E^{1/2}(X_{i, j}) = \psi_1 \psi_2 \sum_{a \in A}E^{1/2}(X_a), 
$$ 
and thus $A$ is $\psi_1 \psi_2$-separating. Note that we have used the fact  
$$
(b_ic_j, X_{i, j}) = 1
$$ twice. 

\end{proof}

Note that  the sets $C_i$ in (\ref{eq:Adecomposition}) may depend on $b_i$, which makes it a more lax assumption than (\ref{eq:sumdecomposition}).

Another crucial ingredient in the model case above is the reduction of the doubling constants (\ref{eq:K_1K_2K}) which is then iterated $\approx \log \log K$ times so that most of the $K$'s  shrink to the scale at which Corollary \ref{corollary:Chang} gives non-trivial results. The next section is fully devoted to a technical lemma which is later used for a similar iterative scheme (see in particular (\ref{eq:FiberingLemmaDoublingConstants})). It seems hard however to guarantee a full analog  (\ref{eq:K_1K_2K}) to hold for arbitrary fibered sets decomposed as (\ref{eq:Adecomposition}). Instead, Bourgain and Chang devised a scheme where the doubling constants are replaced with doubling constants along some graph of density $\delta$. By introducing an additional parameter, the graph density $\delta$ (which, as one can check, never goes below $N^{-o(1)}$), they were able to close the induction.

% Of course, it may well be the case that there are no such sets $A_1, A_2$.  The main insight of Bourgain and Chang is that if $|AA| \leq K|A|$ then, up to logarithmic losses, a subset of $A$ is contained in a product set $A_1A_2$ with $A_1, A_2$ having the desired properties.  Moreover, in order to make further iterations, one has to control the multiplicative structure of        $A_1$ and $A_2$. It turns out that it's hard to guarantee that $A_1$ and $A_2$ have small multiplicative doubling. Instead, another auxiliary parameter $\delta$
% is introduced, such that $A_1$ and $A_2$ are multiplicatively small along some graph of density $\delta$ of order $\log^{-C} |A|$ for some $C > 0$. The precise definition is below.
 
%\begin{defi}[Admissible pairs]
%	A pair of functions $\psi(N, \delta, K), \phi(N, \delta, K)$ is called \emph{admissible} if the following property holds for any two sets $A_1, A_2 \subset \mathcal{R}_0$ of sizes $N_1, %N_2$ with $N := N_1N_2$. 
	
	%If for some graph $G \subset A_1 \times A_2$ with $|G| = \delta N$ holds
	%$$
	%	|A_1 +_{G} A_2| \geq K\sqrt{N_1N_2}	
	%$$ 
	%then there is a subrgraph $G' \subset G$ with
	%$$
	%	|G'| \geq \phi(N, \delta, K)|G|	
	%$$
	%such that for any $a_1 \in A_1$ (resp. $a_2 \in A_2$) the $G'$-neighborhood $G'(a_1) := \{a_2 \in A_2: (a_1, a_2) \in G' \}$ (resp. $G'(a_2)$) is $\psi(N, \delta, K)$-separating.
	%
%
%\end{defi}

%In the next sections we will apply an inductive argument which produces better and better pairs of admissible functions and eventually leads to our goal.

\section{Fibering Lemma} 
This section is devoted to the proof of a structural lemma which is a key ingredient in the inductive step. In fact, the proof works for subsets of linear spaces over any field\footnote{In fact, we need only the structure of a module. The results of this section will be applied later on only for subsets of $\mathbb{Z}^{[n]}$ (viewed as sets in the ambient linear space  $\mathbb{Q}^{[n]}$). We have introduced $F$ here to emphasize that Lemma \ref{lm:FiberingLemma} works equally well when the ambient space is $\mathbb{F}^{[n]}_2$, say.}. We assume the sets in question are subsets of $F^{[n]}$ with a coordinate basis $\{ e_i \}^n_{i=1}$ which we assume fixed. For an index set $I \subset [n]$ there is a natural projection $\pi_I : F^{[n]} \to F^I$ which maps $(x_1, \ldots, x_n)$ to $\sum_{i \in I} x_ie_i$, that is, $\pi_I$ is the projection to the coordinates with indices in $I$. 

When we add two elements $x \in F^I$ and $y \in F^J$ we treat them as elements in the ambient space $F^{I \cup J}$ filling the rest of coordinates with zeroes in the obvious way. When $I \cap J = \emptyset$ we write $x \oplus y$ for the sum to emphasize the orthogonality. This notation extends to sets in the obvious way.

\begin{defi}[Graph fibers]
For a partition $I \cup J = [n]$  a bipartite graph $G \subset X \times Y \subset F^{[n]} \times F^{[n]}$ has a \emph{natural fibering}
$$
	\bigcup_ {(x, y) \in G_I} G_{x, y},
$$
where the \emph{base graph} $G_I$ is defined as
$$
	G_I \equiv \{(\pi_I(u), \pi_I(v)) : (u, v) \in G \} \subset \pi_I(X) \times \pi_I(Y)
$$
and a \emph{fiber graph} $G_{x, y} \subset \pi_J(X) \times \pi_J(Y)$  as
$$
	G_{x, y} \equiv \{(x', y') : (x \oplus x', y \oplus y') \in G  \} \subset \pi_J(X) \times \pi_J(Y).
$$
\end{defi}

We will need another bit of notation to denote fibers. For a set $X$ and $x \in \pi_I(X)$ we write, following the original paper, $X(x)$ for the fiber over $x$. Namely, $X(x)$ is defined as
$$
	X(x) := \{ x' \in \pi_J(X) : x \oplus x' \in X \}.
$$

We will repeatedly use the following "cheap regularity" lemma.
\begin{lemma} \label{lemma:Step0}
  Let $G$ be a graph on $X \times Y$ of size $\delta |X||Y|$. Then there exist $X' \subset X, Y' \subset Y$ and $G' \subset G$ such that 
  \ben
 		|N_{G'}(x)| &\geq& \frac{\delta}{4} |Y| \\
		|N_{G'}(y)| &\geq& \frac{\delta}{4} |X| \\
		|X'| &\geq& \frac{\delta}{2}|X|  \label{eq:Xprimesize} \\
		|Y'| &\geq& \frac{\delta}{2}|Y|  \label{eq:Yprimesize} \\
		|G'| &\geq& \frac{\delta}{2} |X||Y| \label{eq:Gprimesize} 
  \een
  for any $x \in X', y \in Y'$. In particular, for any $A \subset X'$ and $B \subset Y'$
  \ben
  	|(A \times Y') \cap G'| &\geq&\frac{\delta}{4}|Y||A| \\
	|(B \times X') \cap G'| &\geq& \frac{\delta}{4}|X||B|.
  \een
\end{lemma}
\begin{proof}
 Remove from $X$ (resp $Y$) one by one all vertices with degree less than $\delta/4 |Y|$ (resp. $\delta/4 |X|$), until both $X$ and $Y$ contain only vertices of degree at least $\delta/4 |Y|$ (resp. $\delta/4 |X|$) in the remaining graph. Clearly, we cannot remove more than $\delta/2 |X||Y|$ edges no matter how many vertices we remove. Take $X'$ and $Y'$ to be the sets of survived vertices in $X$ and $Y$ respectively and $G' := G \cap (X' \times Y')$.  The bounds (\ref{eq:Xprimesize}) and (\ref{eq:Yprimesize}) follow immediately from (\ref{eq:Gprimesize}).
\end{proof}

Now we can formulate the main lemma of this Section. This is a key ingredient of the original proof and is of independent interest.

\begin{lemma}[Finding a large subset with uniform fibers] \label{lm:FiberingLemma}
	Let $A_1, A_2 \subseteq F^{[n]}$ be subsets of a linear space $V$ over a field $F$ of sizes $N_1, N_2$ respectively. Assume that for some $\delta > 0$ there is a graph $G \subset A_1 \times A_2$ with $|G| = \delta N_1N_2$ such that
	$$
		|A_1 \stackrel{G}{+} A_2| \leq KN_1^{1/2}N_2^{1/2}.	
	$$
	
	Then for any partition $I \cup J = [n]$ there are sets $A'_1 \subset A_1,  A'_2 \subset A_2$ and a subgraph  $G' \subset G$ on $A'_1 \times A'_2$ with the following properties.
		There are numbers $M_1, m_1, M_2, m_2$ and absolute constants $c, C > 0$ with the properties below.
	\begin{enumerate}
		\item(Uniform fiber size) Define $M_1, M_2$ as
		\ben \label{eq:M_1M_2definition}
		|\pi_I(A'_1)| = M_1, \,\,\, |\pi_I(A'_2)| = M_2.
		\een
	 There exist $m_1, m_2 > 0$ such that for any $x \in \pi_I(A'_1), y \in \pi_I(A'_2)$
		\ben \label{eq:FiberingLemmaFiberSizes}
		|A'_1(x)| \approx m_1, \,\,\, |A'_2(y)| \approx m_2
		\een
	  and
	  
	 \ben \label{eq:FiberingLemmaSetSizes}
	 M_1m_1 &\geq& cN_1 \delta^2 \log^{-1}(K/\delta) \\
	 M_2m_2 &\geq& cN_2 \delta^2 \log^{-1}(K/\delta) \\ 
	 m_1, m_2 &\geq& c\delta^{10}K^{-4} \max_{x \in \pi_I(A_1), y \in \pi_I(A_2)} \{ |A_1(x)| + |A_2(y)|\}.
	 \een 
	    
	   \item (Uniform graph fibering) There exist $\delta_1, \delta_2 > 0$ with 
	   \ben \label{eq:FiberingLemmaDeltaSizes}
	   		\delta_1 \delta_2 > c\log^{-3} (\frac{K}{\delta}) \delta.
	   \een
	  such that 	   
	   \ben \label{eq:FiberingLemmaBaseGraphSize}
	   	|G'_I| \geq \delta_1 M_1M_2,
	   \een
	   and for any $(x, y) \in G'_I$ 
	   \ben \label{eq:FiberingLemmaFiberGraphSize}
			|G'_{(x, y)}| \geq \delta_2 m_1m_2.
	   \een
	
	   \item (Bounded doubling) There exist $K_1, K_2 > 0$ 
	   with 
	\ben 	   \label{eq:FiberingLemmaDoublingConstants}
		K_1K_2 \leq CK \log(K)\delta^{-2}.	
	\een 
	such that 
	\ben \label{eq:FiberingLemmaK1}
		|\pi_I(A'_1) \stackrel{G'_I}{+}  \pi_I(A'_2)| = K_1 (M_1M_2)^{1/2}
	\een 
	   and for any $(x, y) \in G'_I$ 
	  \ben \label{eq:FiberingLemmaK2}
		|\pi_J(A'_1) \stackrel{G_{(x, y)}}{+}  \pi_J(A'_2)|	\approx K_2 (m_1m_2)^{1/2}.
	   \een
	
	\end{enumerate}
	
\end{lemma}

\begin{proof}

The proof proceeds in several steps. We will refine $A_1, A_2$ (thus abusing notation) such that eventually all the properties (1)-(3) are satisfied, while keeping track of the losses  with respect to the quantities $\delta, N_1, N_2$ which are fixed and never updated. Without loss of generality we assume $N_1 \geq N_2$. 

The constants $c, C > 0$ are always effective and absolute but may change in the course of the proof. One should think that $c$ is `sufficiently small' and $C$ is `sufficiently large' (though in principle one can evaluate suitable numerical values). 
	
\vspace{1em}
%\noindent \textbf{Step 0. }(Regularizing the graph)
%	
% In what follows it will be convenient to assume that whenever we have a bipartite graph of density $\delta$ on $X \times Y$ (that is, the edge set is of size $\delta|X||Y|$) then 
%$N(x) \geq \delta/4 |Y|$ and $N(y) \geq \delta/4 |Y|$ for any $x \in V_1(G), y \in V_2(G)$. Indeed, we can remove from $X$ (resp $Y$) one by one all vertices with degree less th%an %$\delta/4 |Y|$ (resp. $\delta/4 |X|$), thus removing at most $\delta/2 |X||Y|$ edges, which is clearly permissible. Then we reassign $X$ and $Y$ to these reduced sets. 

%In particular, it follows that $|E(X_1, Y)| \geq \delta' |X_1||Y|$ and $|E(X, Y_2)| \geq \delta' |X_2||Y|$ for any sets $X_1 \subset X$, $Y_2 \subset Y$ with $\delta' = \delta/4$. We then again redefine $\delta := \delta'$ for the sake of notation. %

Applying Lemma \ref{lemma:Step0} with $X = A_1, Y = A_2$ we assume that 
\ben
|A_1| &\geq& \frac{\delta}{2} N_1 \\
|A_2| &\geq& \frac{\delta}{2} N_2 \\
|A_1 \times A_2 \cap G| &\geq&  \frac{\delta}{2} N_1 N_2 \\
\min_{x \in A_1} |x \times A_2 \cap G| &\geq& \frac{\delta}{4} N_2 \label{eq:A_1A_2degree}\\
\min_{y \in A_2} |A_1 \times y \cap G| &\geq& \frac{\delta}{4} N_1 \label{eq:A_2A_1degree}.
\een

\vspace{1em}
\noindent \textbf{Step 1. }(Regularizing the fibers of $A_2$)

Without loss of generality we assume that 
$$
n_1 := \max_{x \in \pi_I(A_1)} |A_1(x)|  \geq \max_{y \in \pi_I(A_2)}  |A_2(y)| .
$$ 
Let $x \in \pi_I(A_1)$ such that $|A_1(x)| = n_1$. We have  $|(x, A_1(x)) \times A_2 \cap G| \geq \frac{\delta}{4} |A_2|n_1$ so we choose using Lemma \ref{lemma:Step0} a subset $A'_2 \subset A_2$ such that 
\beq \label{eq:A_2n_1degree}
	|(x, A_1(x)) \times z \cap G| \geq \frac{\delta}{8} n_1
\eeq
for any $z \in A'_2$.  Also, 
\ben
|A'_2| &\geq& \frac{\delta}{8}|A_2| \geq \frac{\delta^2}{16} N_2 \label{eq:A_2size} \\
|A_1 \times A'_2 \cap G| &\geq& \frac{\delta}{4} N_1|A'_2| \label{eq:A_1A_2density},
\een
since any vertex in $A'_2$ has degree at least $\frac{\delta}{4} N_1$.

We then claim that 
\beq 
 |(x, A_1(x)) \stackrel{G}{+} A'_2| \geq \frac{\delta}{8} n_1 |\pi_I(A'_2)|.
\eeq
Indeed, let 
$$
\{z^{(i)}_I \oplus z^{(i)}_J \}^{|\pi_I(A'_2)|}_{i=1}  \subset A'_2
$$  
be a collection of elements of size $|\pi_I(A'_2)|$ such that $z^{(i)}_I$ are all distinct. 
All the sums 
$$
(x, A_1(x)) + (z^{(i)}_I \oplus z^{(i)}_J)   = (z^{(i)}_I + x) \oplus (z^{(i)}_J + A_1(x))
$$ are distinct, and by (\ref{eq:A_2n_1degree}) at least $\frac{\delta}{8} n_1 |\pi_I(A'_2)|$ of them are in $(x, A_1(x)) \stackrel{G}{+} A'_2$.

Thus,

\ben
\frac{K^2}{\delta}N_2 &\geq& K\sqrt{N_1N_2}  \nonumber \\
						   &\geq& |A_1 \stackrel{G}{+} A_2| \nonumber \\
						   &\geq& |(x, A_1(x)) \stackrel{G}{+} A'_2| \geq \frac{\delta}{8} n_1 |\pi_I(A'_2)| \label{eq:A_2basesize}.
\een
Now define 
\beq
	\bar{A}_2 := \bigcup_{\stackrel {x \in \pi_I(A'_2)}{|A'_2(x)| > 10^{-4} \delta^{5}K^{-2}n_1} } (x, A'_2(x)).
\eeq

Clearly,
\ben
|A'_2 \setminus \bar{A}_2| &\leq& 10^{-4} \delta^{5}K^{-2}n_1 |\pi_I(A'_2)| \nonumber \\
										&\stackrel {(\ref{eq:A_2basesize})}{\leq}& 10^{-3} \delta^{3} N_2 \stackrel{(\ref{eq:A_2size})}{\leq} \frac{\delta}{10} |A'_2| \label{eq:differenceAprime_2Aprime},
\een
so
\ben
|A_1 \times \bar{A}_2 \cap G| \stackrel {(\ref{eq:A_1A_2density})}{\geq} \frac{\delta}{4}N_1 |A'_2|  - \frac{\delta}{10} N_1|A'_2| \geq \frac{\delta}{10} N_1 |A'_2| 
\een

Now, by the dyadic pigeonhole principle there exists $m_2$ with
\beq \label{eq:m2size}
10^{-4}\delta^5 K^{-2}n_1 < m_2 < n_1
\eeq
 such that 
 \beq \label{barAprime_def}
\bar{A}'_2 := \bigcup_{m_2 \leq |\bar{A}_2(x)| < 2m_2} (x, \bar{A}_2(x))
 \eeq
 has size at least $c \log^{-1}{(K/\delta)} |\bar{A}_2|$ for some $c > 0$ (e.g. $10^{-1}$ will do). We conclude
 \ben
|\bar{A}'_2| &\geq& c \frac { |\bar{A}_2|}{\log{(K/\delta)}} \stackrel{(\ref{eq:A_2size}),(\ref{eq:differenceAprime_2Aprime})}{\geq} c \frac{\delta^2}{\log{(K/\delta)}} N_2   \label{eq:Aprime_2size}%\\
%|A_1 \times \bar{A}'_2 \cap G| &\stackrel {(\ref{eq:A_2A_1degree})}{\geq}& c \frac {N_1 |\bar{A}_2|}{\log{(\delta/K)}} \label{eq:A_1Aprime_2density}
 \een
and 
 
% By (\ref{eq:A_2degree}) we conclude
% that
%\beq
%|A_1 \times \bar{A}'_2 \cap G| \geq c \frac {\delta}{\log{\delta/K} } N_1 |\bar{A}'_2|. 
%\eeq

%Now, as we have regularized the fibers, we renew $A_2 := \bar{A}'_2$ for the sake of notation and record the losses (taking into account (\ref{eq:A_2size}))
\ben
|A_1 \times \bar{A}'_2 \cap G| &\stackrel {(\ref{eq:A_2A_1degree})}{\geq}&  \frac{\delta}{4} N_1|\bar{A}'_2| \label{eq:Step1density}\\
%|\bar{A}'_2| &>& c \frac {\delta^3} {\log{(\delta/K)} } N_2 \label{eq:Step1A_2size} \\
|\bar{A}'_2| &\approx& M_2 m_2, \label{eq:Step1M_2m_2}
\een
with $M_2$ defined as
\ben
M_2 &:=& |\pi_I(\bar{A}'_2)|
\een
and $m_2$ defined by (\ref{eq:m2size}), (\ref{barAprime_def}), (\ref{eq:Step1M_2m_2}).

\vspace{1em}
\noindent \textbf{Step 2. }(Regularizing the fibers of $A_1$)

Define $N'_1 := |A_1|, N'_2 := |\bar{A}'_2|$ and record
\ben
  N'_1 &\geq& \frac{\delta}{2} N_1 \\
  N'_2 &>& c \frac {\delta^2} {\log{(K/\delta)} } N_2 \label{eq:Nprime_2size} 
\een

Let 
\ben \label{eq:A_0definition}
A_0 = \bigcup_{\stackrel{x \in \pi_I(A_1)}{|A_1(x)| \leq 10^{-5}\delta^2 K^{-1}m_2}} (x, A_1(x)). 
\een

We want to show that
\ben \label{eq:A0density}
|A_0 \times \bar{A}'_2 \cap G|  \leq \frac{\delta}{40}N_1|\bar{A}'_2| \stackrel{(\ref{eq:Step1density})}{\leq} \frac{1}{10} |A_1 \times \bar{A}'_2 \cap G|.
\een
The argument is similar to the one of Step 1 with $n_1$ replaced by $m_2$. Assume 
$$
|A_0 \times \bar{A}'_2 \cap G|  \geq \frac{\delta}{40}N_1|\bar{A}'_2|.
$$ Then there is $\bar{y} \in \pi_I(\bar{A}'_2)$ such that
\ben
|A_0 \times (\bar{y}, \bar{A}'_2(\bar{y})) \cap G| \geq \frac{\delta}{100}N_1m_2,
\een
since the vertex sets $\{ (y, \bar{A}'_2(y)) : y \in \pi_I(\bar{A}'_2) \}$ are disjoint and are of size $m_2$ each (within a factor of two). Next, let $A'_0 \subset A_0$ be such that
\beq
	|z \times (\bar{y}, \bar{A}'_2(\bar{y})) \cap G|  \geq \frac{\delta}{200} m_2.
\eeq
for each $z \in A'_0$. We clearly have 
\beq
|A'_0| \geq \frac{\delta}{200}N_1.
\eeq

%Namely, we take some $y \in \pi_I(\bar{A}'_2)$ and the fiber 
%$(x, A_2(x)) \subset \bar{A}'_2$. By (\ref{eq:A_2A_1degree}) one has
%\ben
%|A_1 \times (x, \bar{A}'_2(x)) \cap G| \geq \frac{\delta}{4}N_1 |\bar{A}'_2(x)| \geq \frac{\delta}{4}N_1m_2
%\een
%Let $A'_1 \subset A_1$ be such that
%\beq
%	|z \times (x, \bar{A}'_2(x)) \cap G|  \geq \frac{\delta}{8} m_2.
%\eeq
%We have 
%\beq
%|A'_1| \geq \frac{\delta}{16}N_1
%\eeq

Similarly to (\ref{eq:A_2basesize}), denoting  
$$
M := \max_{x \in \pi_I(A'_0)} |A'_0(x)|,
$$ 
we write
\ben
	KN_1 &\geq& K \sqrt{N_1N_2} \geq |A_1 \stackrel{G}{+} A_2| \geq |A'_0 \stackrel{G}{+} (\bar{y}, \bar{A}'_2(\bar{y}))| \nonumber \\
			  & \geq& \frac{\delta}{200} |\pi_I(A'_0)| m_2 \geq \frac{\delta}{200}\frac{|A'_0|}{M}m_2 \\
			  &\geq& \frac{\delta^2}{4\cdot 10^4} \frac {N_1m_2}{M}, \label{eq:Aprime_1basesize}
\een
so
$$
M > 10^{-5}K^{-1}\delta^2 m_2,
$$
which contradicts (\ref{eq:A_0definition}). 

If we now define
\beq
	\bar{A}_1 := \bigcup_{\stackrel {x \in \pi_I(A_1)}{|A_1(x)| > 10^{-5} \delta^{2} K^{-1}m_2} } (x, A_1(x)),
\eeq
then by the preceding discussion and (\ref{eq:Step1density}) we have 
\beq
	|\bar{A}_1 \times \bar{A}'_2 \cap G| \geq \frac{\delta}{8}N_1|\bar{A}'_2|.
\eeq
Since $10^4K^2\delta^{-5}m_2 > n_1 \geq |A_1(x)|$,
by the dyadic pigeonhole principle (since the fibers are disjoint sets of vertices) there exist (the bounds below are somewhat weakened for easier bookkeeping)
\ben
	10^{-5} \delta^{5} K^{-2}m_2 < m_1 < 10^5K^2\delta^{-5}m_2
\een
 such that with 
\beq
	\bar{A}'_1 := \bigcup_{\stackrel {x \in \pi_I(A_1)}{m_1 \leq  |A_1(x)| < 2m_1} } (x, A_1(x)),
\eeq
and some $c > 0$ (say $10^{-2}$) 
\ben
|\bar{A}'_1 \times \bar{A}'_2 \cap G| \geq c\frac{\delta}{\log (K/\delta)}N_1|\bar{A}'_2|.
\een
In particular,
\ben
|\bar{A}'_1| \geq  c\frac{\delta}{\log (K/\delta)}N_1.
\een
Finally, we define
\ben
 M_1 := |\pi_I({\bar{A}'_1})|
\een
so that $|\bar{A}'_1| \approx M_1m_1$ and 
\ben
|\bar{A}'_1(x)| \approx m_1
\een
for each $x \in \pi_I({\bar{A}'_1})$.

%\ben
%|A_1 \tems \bar{A}'_2| &\qeq&  \nonumber \\
%										&\stackrel {(\ref{eq:Aprime_1basesize})}{\leq}& 10^{-3} \delta^{2} N_1 \leq \frac{\delta}{10} |A'_1|,
%\een
\vspace{1em}
\noindent \textbf{Step 3. }(Regularizing the graph fibers)
We renew the definition of $A_1 := \bar{A}'_1$ and $A_2 := \bar{A}'_2$ so that 
\ben 
|A_1 \times A_2 \cap G| &\geq& c\frac{\delta}{\log (K/\delta)}|A_1||A_2|  \label{eq:Step3graphdensity} \\
|A_1| \approx M_1m_1 &\geq& c\frac {\delta}{\log (K/\delta)}N_1  \label{eq:Step3A_1size}\\
|A_2| \approx M_2m_2  &\geq& c \frac {\delta^2}{\log{(K/\delta)}} N_2  \label{eq:Step3A_2size}
\een
and the fibers of $A_1$ and $A_2$ are approximately of size $m_1$ and $m_2$ respectively. 

Recall that for $(x, y) \in \pi_I(A_1) \times \pi_I(A_2)$ we define the fiber graph $G_{(x, y)}$ as
$$
	G_{(x, y)} := \{ (x', y') \in A_1(x) \times A_2(y) : (x \oplus x', y \oplus y') \in G  \}.
$$
In particular, since we have regularized the fibers of $A_1$ and $A_2$, we have
$$
|G_{(x, y)}| \leq 4m_1m_2.
$$

Let
\ben
 G'_I := \{(x, y) \in \pi_I(A_1) \times \pi_I(A_2): |G_{(x, y)}| \geq \frac {c\delta}{16\log (K/\delta)}m_1m_2 \}. 
\een
It follows from (\ref{eq:Step3graphdensity}) that 
\ben
	\sum_{(x, y) \in G'_I} |G_{(x, y)}| \geq  c\frac{\delta}{\log (K/\delta)}|A_1||A_2|.
\een

By dyadic pigeonholing we can find $\delta_2 \gg \frac{\delta}{\log (K/\delta)}$ such that with 
\ben
 \bar{G}'_I := \{(x, y) \in G'_I: \delta_2m_1m_2 \leq |G_{(x, y)}| < 2\delta_2m_1m_2 \}. 
\een
one has
\ben
	\sum_{(x, y) \in \bar{G}'_I} |G_{(x, y)}| \geq  c\frac{\delta}{\log^2 (K/\delta)}|A_1||A_2|.
\een
It then follows that
\ben
| \bar{G}'_I| \geq c \frac{\delta}{\delta_2 \log^2 (K/\delta)} M_1M_2. \label{eq:Step3Gprimesize}
\een

\vspace{1em}
\noindent \textbf{Step 4. }(Regularizing the doubling constant)

Let $(x, y) \in \pi_I(A_1) \times \pi_I(A_2)$. We define 
$$
	K_+(G_{(x, y)}) := \frac{|A_1(x) \stackrel{G_{(x, y)}}{+} A_2(y)|}{\sqrt{|A_1(x)||A_2(y)|}}
$$
to be the normalized doubling constant of the fibers along the fiber graph $G_{(x, y)}$.

Define
$$
H :=  \{(x, y) \in \bar{G}'_I: K_{+}(G_{(x, y)}) > C\log^{3}(K/\delta) \delta^{-10} K \}.
$$

We want to show that $|H| < \frac{1}{10}|\bar{G}'_I|$ provided $C$ is large enough. Indeed, it is trivial that
\beq
|\pi_I(A_1) \stackrel{H}{+} \pi_I(A_2)| \geq \frac{|H|}{\min(|\pi_I(A_2)|, |\pi_I(A_1)|)} \geq \frac {|H|}{\sqrt{M_1M_2}}
\eeq
so
\ben
K \sqrt {N_1N_2} &\geq& |A_1 \stackrel{G}{+} A_2| \\
							&\geq& \min_{(x, y) \in H} |A_1(x) \stackrel{G_{(x, y)}}{+} A_2(y)| |\pi_I(A_1) \stackrel{H}{+} \pi_I(A_2)| \\
							&>&  C\log^{3}(K/\delta) \delta^{-10} K \sqrt{|A_1(x)||A_2(y)|} \frac {|H|}{\sqrt{M_1M_2}} \\
							%&\geq& \frac {\log^{10}(K/\delta) \delta^{-10} K}{\sqrt{M_1M_2}} \\
							&\geq& C\log^{3}(K/\delta) \delta^{-10} K \sqrt{m_1M_1m_2M_2}\frac {|H|}{M_1M_2} \\
							&\stackrel{(\ref{eq:Step3A_1size}),(\ref{eq:Step3A_2size}),(\ref{eq:Step3Gprimesize})}{\geq}& c C K \sqrt {N_1N_2} \frac {|H|}{|\bar{G}'_I|}. 
\een 
Thus, by taking $C$ large enough we can ensure that 
\beq
|H| \leq \frac{1}{10} |\bar{G}'_I|.
\eeq
Denote 
\beq
\bar{G}''_I := \bar{G}'_I \setminus H.
\eeq
By dyadic pigeonholing there is $K_2 \leq C\log^{3}(K/\delta) \delta^{-10} K$ such that 
\beq \label{eq:Step4Gdoubleprimesize}
|\{ (x, y) \in \bar{G}''_I: K_2 \leq K_{+}(G_{(x, y)}) < 2K_2 \}| \geq c \frac{|\bar{G}''_I|}{\log (K/\delta)}.
\eeq
Let $G_{1, 0} \subset \bar{G}''_I \subset \pi_I(A_1) \times \pi_I(A_2)$ be the graph defined by such pairs. We then have
\ben
K \sqrt {N_1N_2} &\geq& |\pi_I(A_1) \stackrel {G_{1, 0}}{+} \pi_I(A_2)|K_2 \sqrt{m_1m_2} \\
						   &\geq& K_+(G_{1, 0}) K_2 \sqrt{m_1M_1m_2M_2} \\
						   &\geq& CK_+(G_{1, 0}) K_2 \delta^{3/2} \frac {\sqrt{N_1N_2}}{\log(K/\delta)},
\een
so
\beq \label{eq:Gdoubleprimedoubling}
K_+(G_{1, 0}) K_2 \leq c\delta^{-3/2} \log(K/\delta) K < C\delta^{-2} \log(K) K.
\eeq

It remains to sum up what we have achieved. By (\ref{eq:Step3Gprimesize}), (\ref{eq:Step4Gdoubleprimesize}) we have
\ben
|\pi_I(A_1)| &=& M_1 \\
|\pi_I(A_2)| &=& M_2 \\
G_{1, 0} &\subset& \pi_I(A_1) \times \pi_I(A_2) \\
|G_{1, 0}| &\geq& c \frac{\delta}{\delta_2 \log^3 (K/\delta)} M_1M_2 
\een
For any $(x, y) \in G_{1, 0}$ we have 
\ben
 A_1(x) &\approx& m_1  \\
 A_2(y) &\approx& m_2   \\
 |G_{(x, y)}| &\approx& \delta_2 m_1m_2 \\
 |A_1(x) \stackrel{G_{(x, y)}}{+} A_2(y)| &\approx& K_2\sqrt{m_1m_2}.
\een
with
\ben
	m_1M_1 &\geq& c\frac{\delta}{\log (K/\delta)}N_1 \\
	m_2M_2 &\geq& c \frac{\delta^2}{\log{(K/\delta)}} N_2  %\\
%	c \delta^{5} K^{-2}m_2 &<& m_1 < cK^2\delta^{-5}m_2	
\een
Defining $\delta_1 := c \frac{\delta}{\delta_2 \log^3 (K/\delta)}$ and $K_1 := K_+({G_{1, 0}})$ we also have
\ben
\delta_1 \delta_2 &>& c\log^{-3} (K/\delta) \delta \\
K_1K_2 &\leq& CK \log(K) \delta^{-2}.
\een
Finally, define 
$$
G' := \{ (x \oplus x', y \oplus y') : (x, y) \in G_{1, 0}; (x', y') \in G_{(x, y)} \}
$$
so that $G'_I = G_{1,0}$ and the proof is finished.

%Finally, 
%\ben
%G' := \bigcup_{(x, y) \in G_{1, 0}} G_{(x, y)} \subset G
%\een
%has size at least 
%\ben
%|G_{1, 0}| \delta_2 m_1m_2 &\geq& c \delta \log^{-3} \log{K/delta} M_1M_2m_1m_2 \\
%											&\geq& c \delta \log^{-5} \log{K\delta} N_1N_2.
%\een

\end{proof}

\section{Iteration scheme}

In this section we will use Lemma \ref{lm:FiberingLemma} in order to setup an iteration scheme. At each step we have a pair of sets $(\mathcal{A}_1, \mathcal{A}_2)$ which correspond to a pair of additive sets $(A_1, A_2) := (\mathcal{P}(\mathcal{A}_1), \mathcal{P}(\mathcal{A}_2))$ and a graph $G$ on $A_1 \times A_2$, together with the data $(N, \delta, K)$ such that:
\ben 
 	|A_1||A_2| &=& N \label{eq:setupNGKdelta1} \\
	|A_1 \stackrel{G}{+} A_2| &\leq& KN^{1/2} \label{eq:setupNGKdelta2} \\
	|G| &\geq& \delta N \label{eq:setupNGKdelta3}
\een

Apart from that, the setup above is equipped with a pair of functions $\psi(N, \delta, K)$, $\phi(N, \delta, K)$ (which are called \emph{admissible} in the original paper). The rather technical definition is below.

\begin{defi}[Admissible pair of functions]
A pair of functions $\psi(N, \delta, K), \phi(N, \delta, K)$ is \emph{admisible} if for arbitrary sets $A_1, A_2 \subset \mathbb{Z}^{[n]}$ and a graph $G$ on $A_1 \times A_2$ satisfying (\ref{eq:setupNGKdelta1})-(\ref{eq:setupNGKdelta3})  the following holds.

There is a graph $G' \subset G$ such that

\begin{enumerate}[(i)]

\item{Graph size is controlled by $\phi$:}
$$
	|G'| \geq \phi(N, \delta, K)
$$

\item{Separation of $G'$-neigborhoods is controlled by $\psi$:\\} 
For any $a_1 \in A_1$ (resp. $a_2 \in A_2$) the $\mathcal{P}$-preimage of the $G'$-neighborhood 
$$
\mathcal{P}^{-1}\left[G'(a_1)\right] := \mathcal{P}^{-1}\left[\{a_2 \in A_2: (a_1, a_2) \in G' \} \right].
$$ (resp. of $G'(a_2)$) is $\psi(N, \delta, K)$-separating.

\end{enumerate}
\end{defi}

Note that by Claim \ref{claim:trivialseparation}, the pair $\psi(N, \delta, K) := N; \phi(N, \delta, K) := \delta N$ is trivially admissible with much room to spare.

The following lemma gives a Freiman-type pair of admissible functions which is better than non-trivial in the regime $K \lll \log N$ and will be used later to bootstrap the argument. 

\begin{lemma}[Freiman-type admissible functions] \label{lm:FreimanAdmissiblePair}
There is an absolute constant $C > 0$ such that the pair of functions
\ben 
		\psi(N, \delta, K) &:=& \min \{ e^{(\frac{K}{\delta})^C}, N \} \label{eq:Freimanadmissiblepsi} \\
		\phi(N, \delta, K) &:=& \left( \frac {\delta}{K} \right)^C N  \label{eq:Freimanadmissiblephi}
\een
is admissible.
\end{lemma} 
\begin{proof}
By the setup, we are given two sets $\mathcal{A}_1, \mathcal{A}_2$ of sizes $N_1, N_2$ and a graph $G$ of size $\delta N_1N_2$ such that
\beq
 |A_1 \stackrel{G}{+} A_2| \leq K \sqrt{N_1N_2} \label{eq:A_1A_2Gdoubling}
\eeq

Assume wlog that $N_1 \geq N_2$  and take $A := A_1 \cup A_2$, which is of size  $ \approx N_1$. Since by (\ref{eq:A_1A_2Gdoubling})
$$
\frac{K^2}{\delta^2} N_2 \geq N_1
$$
we have 
$$
|G| \gg \frac{\delta^3}{K^2}|A|^2
$$
and
$$
|A \stackrel{G}{+} A| \ll K |A|. 
$$
By a variant of the Balog-Szemer\'edi-Gowers theorem (see e.g. \cite{TaoVu}, Excercise 6.4.10) there is $A' \subset A$ such that $|A' + A'| < K'|A'|$ and $|G \cap (A' \times A')| > \delta' N^2_1$ with
\ben
	\delta' > \left( \frac{\delta}{K}\right)^C \label{eq:deltaboundFreiman}\\
	K' < \left( \frac{K}{\delta} \right)^C \label{eq:KboundFreiman}
\een

By Theorem \ref{thm:FreimanLemma} any subset of $A'$ has rank at most $K'$ and by Corollary \ref{corollary:Chang}, the $\mathcal{P}$-preimage of any subset of $A'$ is at most $e^{CK'}$-separating for some $C > 0$.  Thus, taking $G' := G \cap (A' \times A')$ by (\ref{eq:deltaboundFreiman}) and (\ref{eq:KboundFreiman}) we verify that the pair (\ref{eq:Freimanadmissiblepsi}), (\ref{eq:Freimanadmissiblephi}) is admissible.

\end{proof}

The goal is to find a better pair of admissible functions. The lemma below implements the `induction on scales' approach, which allows one to cook up a new pair $\phi_*(N, \cdot, \cdot), \psi_*(N, \cdot, \cdot)$ from a given pair of admissible functions, but taken at the smaller scale  $\approx N^{1/2}$.
 
\begin{lemma} \label{lm:InductionStepLemma}

Let $\psi$ and $\phi$ be an admissible pair of functions. Then for some absolute constant $C > 0$ the pair of functions
\ben \label{eq:admissiblepairinduction}
		\psi_{*}(N, \delta, K) &:=& \max C \psi (N', \delta', K') \psi  (N'', \delta'', K'') \\
		\phi_{*}(N, \delta, K) &:=& \min \phi(N', \delta', K') \phi(N'', \delta'', K''),
\een
is admissible. Here $\min$ and $\max$ is taken over the data $(N', \delta', K'), (N'', K'', \delta'')$ such that
\ben  \label{eq:parameterconstraints}
   c \delta^7 \log^{-22} (K/\delta) N &\leq& N'N'' \leq N \\ \nonumber
   N' + N'' &\leq& C \delta^{-45} K^{11}N^{1/2}	\\ \nonumber
   K'K'' &\leq& C \frac{K \log^8 K }{\delta^{11}} \\ \nonumber
   \delta' \delta'' &\geq& c \log^{-6} (K/\delta) \delta. \nonumber
\een

\end{lemma}

\begin{proof}

 	Let $A_1, A_2 \subset \mathbb{Z}^{[n]}$ of sizes $N_1, N_2$ respectively, $G \subset A_1 \times A_2$ and suppose that the conditions (\ref{eq:setupNGKdelta1})-(\ref{eq:setupNGKdelta3}) are satisfied with parameters $(N := N_1N_2, \delta, K)$. Our ultimate goal is to find a subgraph $G' \subset G$ of size at least $\phi_{*}(N, \delta, K)$ such that the $\mathcal{P}$-preimage of any its neighbourhood is $\psi_{*}(N, \delta, K)$-separating.  In order to achieve this, we will apply Lemma \ref{lm:FiberingLemma} and then use the hypothesis that the pair $\psi, \phi$ is admissible for much smaller sets.	
 	
 	Define a function $f(t)$ for $0 \leq t \leq n$ as
 	$$
		f(t) := \max \{ |A_1(x)| + |A_2(y)| \}
 	$$
 	where the maximum is taken over all $x \in \pi_{[t]}(A_1),  y \in  \pi_{[t]}(A_2)$. Clearly $f$ is decreasing, $f(0) = |A_1| + |A_2| \geq N^{1/2}$ and $f(n) = 0$ so there is $t'$ such that
 	\ben \label{eq:FirstcoordinateSplit}
		f(t') \geq N^{1/4} 	
 	\een
	but
	\ben \label{eq:FirstcoordinateSplit2}
		f(t'+1) < N^{1/4} 	
   \een

	We use the $t'$ defined above for the decomposition $[n] = I \cup J$ with $I := [t']$ and $J := \{t'+1, \ldots, n \}$. We now apply Lemma \ref{lm:FiberingLemma} and get a pair of sets $(A'_1, A'_2)$ together with a graph $G' \subset A'_1 \times A'_2$ such that 
 	\ben
 		A'_1 &=& \bigcup_{x \in \pi_I(A'_1) } (x', A'_1(x')) \\
 		A'_2 &=& \bigcup_{y \in \pi_I(A'_2) } (y', A'_2(y'))
 	\een
 and the fibers $A'_1(x'), A'_2(y')$ together with the fiber graphs $G'_{x, y}$ are uniform as defined in the claim of Lemma \ref{lm:FiberingLemma}. Note that it is possible that $t' = 0$, in which case the sets split trivially with $\pi_I(A'_1) = \pi_I(A'_2) = \{ 0 \}$. 
 
 Using the notation of Lemma \ref{lm:FiberingLemma}  we have 
 \ben
 	|\pi_I(A'_1) \stackrel{G'_I}{+} \pi_I(A'_2)| \leq K_1(M_1M_2)^{1/2}.
 \een
Since $\phi, \psi$ is an admissible pair, there is $G''_I \subset G'_I$ of size at least $\phi(M_1M_2, \delta_1, K_1)$ such that all $\mathcal{P}$-preimages of its vertex neighbourhoods are $\psi(M_1M_2, \delta_1, K_1)$-separating.
Next, for each edge $(x, y) \in G''_I$, since it's a subgraph of $G'_I$, there is a graph $G'_{x, y} \subset \pi_J(A'_1) \times \pi_J(A'_2)$ such that 
\ben
	|A'_1(x) \stackrel{G'_{(x, y)}}{+} A'_2(y)| \leq K_2(m_1m_2)^{1/2}
\een
Again, by admissibility of $\phi, \psi$, there is $G''_{(x, y)} \subset G'_{(x, y)}$ of size at least $\phi(m_1m_2, \delta_2, K_2)$ such that all $\mathcal{P}$-preimages of its vertex neighbourhoods are $\psi(m_1m_2, \delta_2, K_2)$-separating.
 	
Now define $G'' \subset G \cap (A'_1 \times A'_2)$ as 
$$
	G'' := \{(x \oplus x', y \oplus y'): (x, y) \in G''_I, (x', y') \in G''_{(x, y)}  \}.
$$
It is clear by construction that indeed all vertices of $G''$ belong to $A'_1$ and $A'_2$ respectively. Moreover, we have
\ben \label{eq:firstreduction}
	|G''| \geq \phi(M_1M_2, \delta_1, K_1)\phi(m_1m_2, \delta_2, K_2).
\een
 	
Now let's estimate the separating constant for the $\mathcal{P}$-preimage of a neighbourghood $\mathcal{P}^{-1}[G''(u)]$ of some $u \in V(G'')$. Without loss of generality assume that $u \in A'_2$ and $u = v \oplus v'$. We can write
\ben
		G''(u) = \bigcup_{x \in G''_I(v)} \bigcup_{x' \in G''_{(x, v)}(v')}x \oplus x'
\een
Thus,
\ben 
\mathcal{P}^{-1}[G''(u)] = \bigcup_{x \in \mathcal{P}^{-1}[G''_I(v)]} x \left\{ \bigcup_{x' \in \mathcal{P}^{-1}[G''_{(x, v)}(v')]} x' \right\}.
\een
Since $G''_I(v)$ and $G''_{(x, v)}(v')$ are orthogonal as linear sets we conclude that $(x, x') = 1$ for $x \in \mathcal{P}^{-1}[G''_I(v)]$ and $x' \in \mathcal{P}^{-1}[G''_{(x, v)}(v')]$.  Thus, by Claim \ref{claimPsi} and the admissibility of $\phi, \psi$ applied to $G''_I$ and $G''_{(x, v)}$ we conclude that $\mathcal{P}^{-1}[G''(u)] $ is at most $\psi(M_1M_2, \delta_1, K_1)\psi(m_1m_2, \delta_2, K_2)$-separating.

We now record the bounds for $\delta_2, K_2, m_1, m_2$ given by Lemma \ref{lm:FiberingLemma}. We have 
\ben
	\delta_1\delta_2 &>& c \log^{-3}(K/\delta) \delta. \label{eq:delta_2}\\
	K_1 K_2 &<&  CK \log K \delta^{-2}   \label{eq:K_2} \\
	M_1m_1 &>& c \delta^2 \log^{-1} (K/\delta) N_1 \label{eq:M_1m_1} \\
	M_2m_2 &>& c \delta^2 \log^{-1} (K/\delta) N_2 \label{eq:M_2m_2} \\
	m_1, m_2 &>& c \delta^{10} K^{-4} N^{1/4}
\een

In particular, we have
\ben \label{eq:M_1M_2upperbound}
 M_1M_2 < \frac {N_1N_2}{m_1m_2} < c \delta^{-20}K^{8} N^{1/2}.
\een

We have thus verified the claim of the Lemma save for the fact that $m_1m_2$ may well be larger than $N^{1/2}$. In order to further reduce the size we apply Lemma \ref{lm:FiberingLemma} again for each pair of sets $(A'_1(x), A'_2(y))$ such that $(x, y) \in G'_I$, stripping off only a single coordinate as explained below. Assume the base point $(x, y)$ is fixed henceforth. 

\begin{rem}
The reader may keep in mind the following model case: $N_1 = N_2 = N^{1/2}$ and $A_1, A_2$ are three-dimensional arithmetic progressions
$$
\{ [0,N_1^\alpha]e_1 + [0, N_1^{1-2\alpha}]e_2 +  [0, N_1^\alpha]e_3  \},
$$
for some small $\alpha > 0$. In this case $t'=1, M_1 M_2 \approx N^\alpha$  and $m_1m_2 \approx N^{1-\alpha}$. However, the fibers $A_1(x)$ and $A_2(y)$ are heavily concentrated on the second coordinate which makes the separating constant of the $\mathcal{P}$-preimage small.
\end{rem}
\vskip 1em

We split the coordinates $J = \{t'+1, \ldots, n \}$ into 
$$
I' = \{ t' + 1 \}
$$ and 
$$
J' = \{t'+2, \ldots, n \}
$$ and apply  Lemma \ref{lm:FiberingLemma} for such a decomposition,  the pair $(A'_1(x), A'_2(y))$ and the graph $G_{(x, y)}$. 
We then get 
$$
A''_1 \subset  A'_1(x), \,\,\, A''_2 \subset A'_2(y)
$$ such that
\ben
	A''_1 &=& \bigcup_{w \in \pi_{I'}(A''_1)} (w, A''_1(w)) \\
	A''_2 &=& \bigcup_{z \in \pi_{I'}(A''_2)} (z, A''_2(z))
\een
and the fibers $A''_1(w)$ and $A''_2(z)$ are of approximately the same size $l_1$ and $l_2$ respectively. Note again, that, say,  the fiber $A''_1(w)$ may be trivial (e.g. $\{ 0\}$), which simply means that $l_1 \approx 1$.  

Next, we have a graph $\mathcal{K} \subset A''_1 \times A''_2$  with uniform fibers as defined in Lemma \ref{lm:FiberingLemma}.  Note that $l_1, l_2$ and $\mathcal{K}$ depend on the base point $(x, y)$ which we assume is fixed. 

The graph $\mathcal{K}$ splits into the one-dimensional base graph $\mathcal{K}_{I'} \subset \pi_{I'}(A''_1) \times \pi_{I'}(A''_2)$ and fiber graphs $\mathcal{K}_{w, z}$ such that for $(w, z) \in \mathcal{K}_{I'}$ 
\ben
|A''_1(w) +_{\mathcal{K}_{w, z}} A''_2(z)| \leq K_3(l_1l_2)^{1/2},
\een
with
\ben
|A''_1(w)| &\approx& l_1 \\
|A''_2(z)| &\approx& l_2 \\
|\mathcal{K}_{w, z}| &\geq& \delta_3 \l_1l_2.
\een

The parameters $l_1, l_2, \delta_3, K_3$  as well as the sizes of $\mathcal{K}_{I'}$ and $\mathcal{K}_{w, z}$ are controlled by Lemma \ref{lm:FiberingLemma}. By induction hypothesis, for each such a graph $\mathcal{K}_{w, z}$ there is a subgraph $\mathcal{K}'_{w, z} \subset \mathcal{K}_{w, z}$ with
\ben
|\mathcal{K}'_{w, z}| \geq \phi(l_1l_2, \delta_3, K_3)
\een
such that the $\mathcal{P}$-preimage of each neighborhood of $\mathcal{K}'_{w, z}$ is $\psi(l_1l_2, \delta_3, K_3)$-separating. Define $\mathcal{K}' \subset \mathcal{K}$ as
\ben
	\mathcal{K}' := \{(w \oplus w', z \oplus z'): (w, z) \in \mathcal{K}_{I'}, (w', z') \in \mathcal{K}'_{w, z} \}. 
\een
The size of $\mathcal{K}'$ is clearly at least $|\mathcal{K}_{I'}|\phi(l_1l_2, \delta_3, K_3)$. Next, the set of vertices of $\mathcal{K}_{I'}$ all lie in a one-dimensional affine subspace, so combining (\ref{eq:oneprimeseparation}) and Claim \ref{claimPsi} one concludes that the $\mathcal{P}$-preimage of each neighborhood of $\mathcal{K}'$ is 
$C\psi(l_1l_2, \delta_3, K_3)$-separating with some absolute constant $C > 0$ (for the additive energy one can take $C = 6$ as in (\ref{eq:oneprimeseparation})). Summing up, we conclude
that 
\ben \label{eq:step2psiphi}
	\psi_{x, y} := C\psi(l_1l_2, \delta_3, K_3); \,\,\, \phi_{x, y} := |\mathcal{K}_{I'}|\phi(l_1l_2, \delta_3, K_3)  
\een
is admissible for the pair of sets $(A'_1(x), A'_2(y))$ and the graph $G_{x, y}$. In turn, substituting $\psi_{x, y}$ and $\phi_{x, y}$ into the argument leading to (\ref{eq:firstreduction}) and Claim \ref{claimPsi}, one concludes that
\ben  \label{eq:admissiblepairs2}
	\psi(M_1M_2, \delta_1, K_1)   \max_{(x, y) \in G'_I}\psi_{x, y}  ;\,\,\,\,\, 	\phi(M_1M_2, \delta_1, K_1)  \min_{(x, y) \in G'_I} \phi_{x, y}
\een
is admissible for $(A_1, A_2)$ and the graph $G$. It remains to check that the quantities (\ref{eq:admissiblepairs2}) can indeed be bounded as (\ref{eq:admissiblepairinduction}). By saying that (\ref{eq:admissiblepairs2}) is admissible we mean that we can find a subgraph of $G$ of size at least  
$$
\psi(M_1M_2, \delta_1, K_1) \cdot  \max_{(x, y) \in G'_I}\psi_{x, y}
$$
such that the separating factors are most
$$
\phi(M_1M_2, \delta_1, K_1) \cdot \min_{(x, y) \in G'_I} \phi_{x, y}.
$$
Note that the quantities (\ref{eq:admissiblepairs2}) do depend on the structure of $A_1, A_2$. We are going to show, however, that they are uniformly bounded by (\ref{eq:admissiblepairinduction}) which are functions of $(N, \delta, K)$ only.

%Now we will use the bounds for $l_1, l_2, \delta_3, K_3, |K_{I'}|$ given by Lemma \ref{lm:FiberingLemma} to get meaningful information out of (\ref{eq:step2psiphi}).

First, since $(x, y) \in G'_I$ we have by (\ref{eq:FiberingLemmaDeltaSizes})
\ben
	 \delta_3 > c \log^{-3}(K_2/\delta_2) \delta_2.
\een
By (\ref{eq:FiberingLemmaDoublingConstants}) and (\ref{eq:FiberingLemmaDeltaSizes})
\ben
	\frac {K_2} {\delta_2} < \frac{CK \log(K) \log^3(K/\delta)} {\delta^{3}}
\een
so
\ben
 \log (K_2/\delta_2) < C\log(K/\delta).
\een
and
\ben \label{eq: delta_1delta_3}
 \delta_1 \delta_3 > c \log^{-3} (K/\delta) \delta_1 \delta_2 > c \log^{-6} (K/\delta) \delta.
\een
Next, by (\ref{eq:FiberingLemmaDoublingConstants})
\ben
 K_3  \leq C K_2 \log (K_2) \delta_2^{-2}.
\een
and by (\ref{eq:FiberingLemmaDeltaSizes})
\ben
\delta_2 &>& c \log^{-3}(K/\delta) \delta \label{eq:delta2delta}\\
K_2 &<& C  \delta^{-2} K \log K
\een
so
\ben
\log(K_2) \delta_2^{-2} &\leq& C \log^7 (K/\delta) \delta^{-2} \label{eq:logK2delta2}  \\
									&=& C (\delta^7 \log^7 (K/\delta)) \delta^{-9} < C \log^7 K \delta^{-9} \nonumber
\een
and
\ben \label{eq:K_1K_3bound}
K_1 K_3 \leq C K_1 K_2 \log (K_2) \delta_2^{-2} \stackrel{(\ref{eq:K_2}), (\ref{eq:logK2delta2})}{\leq} C \frac{K \log^8 K }{\delta^{11}}
\een

We turn to $|\mathcal{K}_{I'}|$ and $l_1l_2$. We have by (\ref{eq:FiberingLemmaSetSizes}), (\ref{eq:FiberingLemmaDeltaSizes}), (\ref{eq:FiberingLemmaBaseGraphSize}), (\ref{eq:FiberingLemmaFiberGraphSize}) that
\ben
|\mathcal{K}_{I'}|l_1l_2 &\geq& c \log^{-3}(K_2/\delta_2) \delta_2 (\delta_2^4 \log^{-2}(K_2/\delta_2)) |A_1(x)||A_2(y)| \nonumber \\
					&\geq& c\log^{-5} (K/\delta) \delta_2^5 m_1m_2 \stackrel {(\ref{eq:delta2delta})}{\geq} c \log^{-20} (K/\delta) \delta^5 m_1m_2  \label{eq:K_Il_1l_2} 
\een

Define 
\ben
 N'' := \min \{ N^{1/2},  \max\{ l_1l_2, c \log^{-20} (K/\delta) \delta^5 m_1m_2 \} \}
\een
By our choice of $t'$ it follows that $l_1l_2 \leq N''$. Thus, we may assume
\ben
\frac{l_1l_2}{N''}\phi(N'', \delta_3, K_3) \leq  \phi(l_1l_2, \delta_3, K_3).
\een
Indeed, the function $\phi(\cdot, \delta_3, K_3)$ may always be taken sublinear since one can take a sparser graph if needed \footnote{In what follows  $\phi(N, \delta, K)$ eventually will be of the form $N^{1-\tau + o(1)}$ for some $\tau > 0$. }.  
For the same reason, defining 
\ben
N' := \frac{M_1M_2l_1l_2}{N''}|\mathcal{K}_{I'}|
\een
we have by (\ref{eq:K_Il_1l_2}) that $M_1M_2 \leq N'$ and
\ben
\frac{M_1M_2}{N'}\phi(N', \delta_1, K_1) \leq  \phi(M_1M_2, \delta_1, K_1),
\een 
so
\ben \label{eq:phiNprimesbound}
	\phi(N', \delta_1, K_1) \phi (N'', \delta_3, K_3) &\leq& \frac{N'}{M_1M_2}\phi(M_1M_2, \delta_1, K_1) \frac{N''}{l_1l_2}\phi(l_1l_2, \delta_3, K_3) \nonumber \\
																		 &\stackrel {(\ref{eq:step2psiphi})}{=}&  \phi(M_1M_2, \delta_1, K_1) \phi_{x, y}
\een
On the other hand, 
\ben \label{eq:Nprimesproduct}
N'N'' &=& M_1M_2l_1l_2 |\mathcal{K}_{I'}| \stackrel {(\ref{eq:K_Il_1l_2})}{\geq} c\log^{-20} (K/\delta) \delta^5 M_1M_2m_1m_2 \\
		&\stackrel {(\ref{eq:FiberingLemmaSetSizes})}{\geq}& c  \delta^7 \log^{-22} (K/\delta) N.
\een
Also, since
\ben
	m_1m_2 > c \delta^{20} K^{-10} N^{1/2},   \nonumber
\een
it follows that
\ben
	c \delta^{45} K^{-11}N^{1/2} \leq N'' \leq N^{1/2} \nonumber
\een
and, since $N'' N' \leq N$,
\ben
  N' \leq C \delta^{-45} K^{11}N^{1/2}, \nonumber
\een
so
\ben \label{eq:Nprimessum}
 N' + N'' \leq C \delta^{-45} K^{11}N^{1/2}.
\een

We now have all the estimates to finish the proof. The bounds (\ref{eq: delta_1delta_3}), (\ref{eq:K_1K_3bound}), (\ref{eq:Nprimesproduct}), (\ref{eq:Nprimessum}) verify that the parameters 
\ben 
\delta' &:=& \delta_1, \,\,\, \delta'' := \delta_3 \nonumber \\
K' &:=& K_1, \,\,\, K'' := K_3 \nonumber \\
\een
and $N', N''$ indeed satisfy the constraints (\ref{eq:parameterconstraints}).  Next, it is trivial that $\psi(\cdot, \delta, K)$ can be taken monotone increasing in the first argument, so by (\ref{eq:M_1M_2upperbound}) and (\ref{eq:FirstcoordinateSplit2})
\ben 
		\psi_{*}(N, \delta, K) 	&\geq& C \psi(\max \{N^{1/2}, \frac{N}{m_1m_2} \}, \delta_1, K_1)  \psi( \min \{ N^{1/2}, m_1m_2 \}, \delta_3, K_3) \nonumber \\
					&\geq&  \psi(M_1M_2, \delta_1, K_1) \psi_{x, y} \label{eq:psistargeneralbound}
\een

Also, (\ref{eq:phiNprimesbound}) and (\ref{eq:admissiblepairs2}) verify that 
\ben \label{eq:phistargeneralbound}
\phi_{*}(N, \delta, K) &\leq& \phi(N', \delta_1, K_1) \phi (N'', \delta_3, K_3)  \nonumber \\
																	 &\leq& \phi(M_1M_2, \delta_1, K_1) \phi_{x, y}. \label{eq:phistargeneralbound}
\een
 It follows that the pair $(\psi_{*}, \phi_{*})$ is indeed admissible since (\ref{eq:psistargeneralbound}) and (\ref{eq:phistargeneralbound}) hold for all base points $(x, y) \in G'_I$ and thus uniformly bound (\ref{eq:admissiblepairs2}). 
\end{proof}
 
\section{A better admissible pair}

With Lemma \ref{lm:InductionStepLemma} at our disposal we can start with the data $(N, \delta, K)$ and  reduce the problem to the case of smaller and smaller $N$ and $K$ with reasonable losses in $\delta$. The process can be described by a binary a tree where each node with the data $(N, \delta, K)$ splits into two children with the attached data being approximately equal $(N^{1/2}, \delta', K')$ and $(N^{1/2}, \delta', K'')$, with $K'K''$ roughly equal to $K$ and $\delta'\delta''$ roughly equal to $\delta$. Thus, when the height of the tree is about $\log \log K$, the $K$'s in the most of the nodes should be small enough so that Lemma \ref{lm:FreimanAdmissiblePair} becomes non-trivial. Going from the bottom to the top we then recover an improved admissible pair of functions at the root node. 

\begin{lemma} \label{lm:beteradmissiblepair}
For any $\gamma > 0$ there exists $C(\gamma) > 0$ such that the pair
\ben
	\phi(N, \delta, K) := \left( \frac{\delta}{K}\right)^{C \log \log (K/\delta)} N \label{eq:betterphi}\\
	\psi(N, \delta, K) := \exp (\log (K/\delta)^{C/\gamma}) N^{\gamma} \label{eq:betterpsi}
\een  
is admissible. 
\end{lemma}

\begin{proof}
Take an integer $t = 2^l$ to be specified later ($l$ is going to be the height of the tree and $t$ the total number of nodes). Each node then has an index $\nu \in \{0, 1\}^l$.
We start with an admissible pair $\phi_0, \psi_0$ given by Lemma \ref{lm:FreimanAdmissiblePair} at the bottom-most level.
Going recursively from the leaves to the root, we have by Lemma  \ref{lm:InductionStepLemma} that for the levels $i = 1, \ldots, l$ the pairs

\ben \label{eq:admissiblepairinductionleveli}
		\psi_{i} &:=& \max C \psi_{i-1} (N', \delta', K') \psi_{i-1}  (N'', \delta'', K'') \\
		\phi_{i} &:=& \min \phi_{i-1}(N', \delta', K') \phi_{i-1}(N'', \delta'', K''),
\een
are admissible (with the $\max$ and $\min$ taken over the set of parameters constrained by (\ref{eq:parameterconstraints})). Thus,  at the root node we have the admissible pair $\psi := \psi_{l-1}, \phi := \phi_{l-1}$ given by
\ben
	\psi(N, \delta, K) &:=& C^{2^l} \prod_{\nu \in \{0, 1 \}^l} \psi_0(N'_\nu, \delta'_\nu, K'_\nu)	\label{eq:betterpsiproductformula}\\
	\phi(N, \delta, K) &:=& \prod_{\nu \in \{0, 1 \}^l} \phi_0(N_\nu, \delta_\nu, K_\nu) \label{eq:betterphiproductformula} \\
\een
for some data $(N_\nu, \delta_\nu, K_\nu)$ and (possibly different) $(N'_{\nu}, \delta'_{\nu}, K'_{\nu})$ at the leaf nodes of the tree which attain the respective maxima and minima. For non-leaf tree nodes $\nu$, denoting by $\{\nu, 0\}$ and $\{\nu, 1\}$ the left and right child of $\nu$ respectively, one has
\ben  \label{eq:parameterconstraintslevel}
   c \delta_\nu^7 \log^{-22} (K_\nu/\delta_\nu) N_\nu &\leq& N_{\nu, 0}N_{\nu, 1} \leq N_\nu \\ \nonumber
   N_{\nu, 0} + N_{\nu, 1} &\leq& C \delta_\nu^{-45} K_\nu^{11}N_\nu^{1/2}	 \label{eq:Nsquarerootsplitting} \\
   K_{\nu, 0}K_{\nu, 1} &\leq& C \frac{K_\nu \log^8 K_\nu }{\delta_\nu^{11}} \\ \nonumber
   \delta_{\nu, 0} \delta_{\nu, 1} &\geq& c \log^{-6} (K_\nu/\delta_\nu) \delta_\nu. \nonumber
\een
and similarly for $(N'_{\nu}, \delta'_{\nu}, K'_{\nu})$.

In what follows we assume that $N$ is large enough so that $\log K_\nu > C$ and $\log(\delta_\nu^{-1}) > c^{-1}$ and the constants $C, c$ can be swallowed by an extra power of $\log (K/\delta)$.

We have
\beq
 \log \frac{K_{\nu, 0}}{\delta_{\nu, 0}} +  \log \frac{K_{\nu, 1}}{\delta_{\nu, 1}} < 15 \log \frac{K_\nu}{\delta_\nu}
\eeq
so for an arbitrary $1 < l' \leq l$
\beq \label{eq:Kdeltanubound}
   \max_{  \nu \in \{0, 1\}^{l'} } \log \frac{K_\nu}{\delta_\nu} <  15^{l'} \log \frac{K}{\delta}.
\eeq
Next, it follows (see the original paper for more details) from (\ref{eq:Kdeltanubound}) that 
\beq \label{eq:deltanuproductbound}
 \prod_{\nu \in \{0, 1\}^{l'}} \delta_\nu > 15^{-6l' \cdot 2^{l'}} \log (\frac{K}{\delta})^{-6\cdot 2^{l'}} \delta
\eeq
and
\beq \label{eq:Knuproductbound}
  \prod_{\nu \in \{0, 1\}^{l'}} K_\nu < 15^{50 \cdot l' 2^{l'}} \log (\frac{K}{\delta})^{50 \cdot 2^{l'}} \delta^{-7l'}K
\eeq
and
\beq \label{eq:Nnuproductbound}
	\prod_{\nu \in \{0, 1\}^{l'}} N_\nu > 15^{-100 \cdot l' 2^{l'}} \left(\log \frac{K}{\delta} \right)^{-200 \cdot 2^{l'}} \delta^{30 l'} N.
\eeq

Taking 
$$
l := \lfloor \log \log (K/\delta) \rfloor
$$ and substituting (\ref{eq:deltanuproductbound}), (\ref{eq:Knuproductbound}), (\ref{eq:Nnuproductbound}) into (\ref{eq:betterphiproductformula}) and (\ref{eq:Freimanadmissiblephi}) we get
$$ 
 \prod_{\nu \in \{0, 1\}^{l}} (\frac{\delta_\nu}{K_\nu})^C N_\nu \geq \left( \frac{\delta}{K}\right)^{C \log \log (K/\delta)} N
$$
for some suitable $C > 0$. The elementary but a bit tedious calculations can be found in the original paper. We note however, that it is natural to expect that the resulting function should look like (\ref{eq:betterphi}). Ignoring $\delta$'s, we loose at each node at most by a multiplicative factor of $ \log^{-C} K$, totaling to the $(\log K)^{-C2^l}$ loss, which is approximately $K^{-C \log \log K}$.

We now turn to $\psi$. Again, we omit the details but only sketch the main idea of the calculation. For the sake of notation we use again $(N_\nu, \delta_\nu, K_\nu)$ instead of $(N'_\nu, \delta'_\nu, K'_\nu)$. The bounds above, however, still hold.

We split the data $(N_\nu, \delta_\nu, K_\nu)$ into two parts, $I \cup J = \{0, 1 \}^l$, such that
$$
	\frac{K_\nu}{\delta_\nu} < A : \nu \in I
$$
and
$$
	\frac{K_\nu}{\delta_\nu} \geq A : \nu \in J,
$$
with the threshold $A$ specified in due course. 

By (\ref{eq:deltanuproductbound}) and (\ref{eq:Knuproductbound}) it is easy to see that $|J|$ is rather small:
\ben \label{eq:AJbound}
 A^{|J|} \leq \prod_{\nu \in \{0, 1 \}^l}  \frac{K_\nu}{\delta_\nu} < 15^{100 \cdot l 2^{l}} \log (\frac{K}{\delta})^{100 \cdot 2^{l}} \delta^{-10l}K.
\een
Set $t := 2^l$ and take
$$
 \log A := 10^3 \gamma^{-1} l = \frac{10^3 \log \log (K/\delta)}{\gamma}.
$$
It follows from (\ref{eq:AJbound}) that
$$
|J| \log A  \leq 500 l t + 100 t l + 10 t l + t < 10^3 l t,
$$
so 
$$
\frac{|J|}{t} < \frac{10^3 l}{\log A} = \gamma.
$$

By (\ref{eq:parameterconstraintslevel}) we have that at the bottom level each $N_\nu \approx N^{\frac{1}{2^l}}$, so we can estimate by Lemma \ref{lm:FreimanAdmissiblePair} (ignoring the logarithmic losses at each node, which one checks are acceptable)
$$
\psi \leq C^t \prod_{\nu \in \{0, 1\}^l} \min \{e^{(\frac{K_\nu}{\delta_\nu})^C}, N_\nu \}  \approx e^{|I|A^C}N^{\frac{|J|}{t}} < e^{tA^C}N^{\frac{|J|}{t}} < \exp (\log (K/\delta)^{C/\gamma}) N^{\gamma}.
$$

\end{proof}

\section{A strong admissible pair}

Finally, in this section we will use Lemma \ref{lm:beteradmissiblepair} to get an even better pair of admissible functions. 
\begin{lemma} \label{lm:finaldamissiblepairs}
Given $0 < \tau, \gamma < 1/2$ there exist positive constants $A_i(\tau, \gamma), B_i (\tau, \gamma), i= 1, 2, 3$ such that 
\ben
  \phi := K^{-A_1} \delta^{A_2 \log \log N} e^{A_3 (\log \log N)^2 } N^{1 - \tau} \label{eq:bestphi} \\
  \psi :=   K^{B_1} \delta^{-B_2 \log \log N} e^{-B_3 (\log \log N)^2 } N^{\gamma} \label{eq:bestpsi}
\een
are admissible. 

\end{lemma}
\begin{proof}

The strategy of the proof is as follows. We start with already a not-so-bad admissible pair given by Lemma \ref{lm:beteradmissiblepair} and improve it by repeated application of Lemma \ref{lm:InductionStepLemma}. 

The idea is that first we find a fixed threshold $\bar{N}(\tau, \gamma)$ such that the pair (\ref{eq:bestphi}), (\ref{eq:bestpsi}) is either trivial or worse than that given by Lemma \ref{lm:beteradmissiblepair} if $N \leq \bar{N}$. One can achieve this by fine-tuning the constants $A_1, B_1$. 

After such a bootstrapping we have an intermediate admissible pair, say $(\phi',\psi')$, defined by 
(\ref{eq:bestphi}), (\ref{eq:bestpsi}) if $N \leq \bar{N}$ 
and 
by (\ref{eq:betterphi}), (\ref{eq:betterpsi}) otherwise.

Next, we use Lemma \ref{lm:InductionStepLemma} and prove that for a suitable choice of $A_2, A_3, B_2, B_3$ the following induction step holds.

Assume $(\phi_X,\psi_X)$ is an admissible pair defined as (\ref{eq:bestphi}), (\ref{eq:bestpsi}) if $N \leq X$ and by (\ref{eq:betterphi}), (\ref{eq:betterphi}) otherwise. Then the pair $(\phi_{X^2},\psi_{X^2})$ defined as (\ref{eq:bestphi}), (\ref{eq:bestpsi}) if $N \leq X^2$ and by (\ref{eq:betterphi}), (\ref{eq:betterpsi}) otherwise, is also admissible. 

Iterating, one then concludes that (\ref{eq:bestphi}), (\ref{eq:bestpsi}) is admissible for all $(N, \delta, K)$.

%Using the tree structure as in the proof of Lemma \ref{lm:beteradmissiblepair} we start with the %root node with the data $(N, \delta, K)$ and estimate
%the admissible pair of functions at each node as the admissible functions evaluated at the data %attached to the children. At the tree 
%height $t$ the parameter $N_\nu$ of the triple $(N_\nu, \delta_\nu, K_\nu)$ is at scale %$N^{1/2^t}$. Taking $t$ is as large as $\log \log N$, we can get below any fixed threshold.

%In turn, if the pair (\ref{eq:bestphi}), (\ref{eq:bestpsi}) is admissible at some level $t$ (so that %$N$ is below a fixed threshold,  $\bar{N}(\tau, \gamma)$) with some set of constants $A_i, %B_i$ (which may depend on $\bar{N}$), it is sufficient (by induction) to check by plugging %(\ref{eq:bestphi}), (\ref{eq:bestpsi}) into Lemma \ref{lm:InductionStepLemma} that the pair  %(\ref{eq:bestphi}), (\ref{eq:bestpsi}), with the same set constants $A_i, B_i$ is admissible at the %level $t-1$. 

%Assume now that we are at the lowest level of the tree so that $N \leq \bar{N}$ and $\log N %\approx \log  \bar{N}$. $\bar{N}$ is fixed. We fix only $A_1$ and $B_1$ and define $A_2, B_2, %A_3, B_3$ in due course (it will be important that $A_3 > A_2 > A_1$ and $B_3 > B_2 > B_1$) %. 
Let us see how this induction scheme can be implemented. Starting with (\ref{eq:bestphi}), (\ref{eq:bestpsi}) of Lemma \ref{lm:beteradmissiblepair} we should take $A_1, B_1$ and a threshold $\bar{N}(\delta, \gamma)$ such that for $N \leq \bar{N}$
\ben
	\left( \frac{\delta}{K}\right)^{C \log \log (K/\delta)} N > (\ref{eq:bestphi}) \\
	\exp (\log (K/\delta)^{C/\gamma}) N^{\gamma} < (\ref{eq:bestpsi})
\een
For (\ref{eq:bestphi}) it's sufficient to take $A_1 = C\log \log \bar{N}$ with some $C(\tau) > 0$. (\ref{eq:bestpsi}) is more tricky, as later on it will be important that $B_3 > B_2 > B_1$.
It sufficies to guarantee that 
\ben \label{eq:Kboundexponencial}
 \log \left( \frac{K}{\delta} \right)^{\frac{ C}{\gamma}} < \frac{\gamma}{4} \log N
\een
and
\ben \label{eq:B_3boundexponenial}
e^{B_3 (\log \log N)^2} < N^{\frac{\gamma}{2}}.
\een

The bound (\ref{eq:Kboundexponencial}) does not hold only if $K/\delta$ is rather large,
$$
\frac{K}{\delta} > e^{\log^{c\gamma} N}
$$
for some $c(C, \gamma) > 0$.  In this case it suffices to take $B_1$ so large that (\ref{eq:bestpsi}) $> N$ and thus trivially admissible. To this end it suffices to take
$$
B_1 := (\log \bar{N})^{1 - c\gamma}
$$
and make the constraint that say
$$
B_3, B_2 < 10 B_1  \log \log {\bar{N}}.
$$ 

Summing up, we have found some fixed threshold $\bar{N}(\tau, \gamma)$  at which (\ref{eq:bestphi}), (\ref{eq:bestpsi}) become admissible, with fixed $A_1, B_1$ and some freedom to define the constants $A_2, B_2, A_3, B_3$.

Now, assuming that $N', N''$ are at the scale so that (\ref{eq:bestphi}), (\ref{eq:bestpsi}) are admissible with the data $(N', \delta', K'); (N'', \delta'', K'')$ we will show that  (\ref{eq:bestphi}), (\ref{eq:bestpsi}) are also admissible for the data $(N, \delta, K)$ with $N \approx N'N''$.

Assuming $B_1$ (or $\bar{N}$) is large enough we may assume that 
\ben \label{eq:Kdeltasmall}
\frac{K}{\delta} < N^{10^{-3}}.
\een
as otherwise (\ref{eq:bestpsi})  $ > N$ which is trivially admissible.  

We need to estimate 
$$
 \psi (N', \delta', K') \psi  (N'', \delta'', K'') 
$$
from above and 
$$ 
 \phi(N', \delta', K') \phi(N'', \delta'', K''),
$$
from below in order to verify that (\ref{eq:bestphi}), (\ref{eq:bestpsi}) are admissible for $(N, \delta, K)$. By (\ref{eq:Kdeltasmall}), the constraints (\ref{eq:parameterconstraints}) can be relaxed to
\ben
N \geq N'N'' &>& N\left( \frac{\delta}{\log N}\right)^{40} > N^{99/100} \label{eq:N_1N_2simple} \\
N' + N'' &<& N^{1/2} \left(\frac{K}{\delta} \right)^{40} < N^{1/2 + 1/40}  \label{eq:N_1plusN_2simple}\\
\delta' \delta'' &>& \frac{\delta}{\log^6 N} \label{eq:deltadeltaprimesimple}\\
K'K'' &<& \delta^{-10} (\log N)^{10}K \label{eq:KprimeKprime}
\een 

From (\ref{eq:N_1N_2simple}) and (\ref{eq:N_1plusN_2simple}) we have (with room to spare)
\beq
N^{1/2 - 1/20} < N', N'' < N^{1/2 + 1/20}
\eeq
and assuming $N$ is large enough
\beq
\frac{99}{100} \log \log N < \log \log N', \log \log N'' < \log \log N + \frac{20}{11}.
\eeq
With the constraints above, it suffices to verify (writing $\l \l$ for $\log \log $ like in the original paper) that
\ben \label{eq:bestphisubstitution}
(K'K'')^{-A_1} (\delta')^{A_2 \l \l N'} (\delta'')^{A_2 \l \l N''} e^{A_3[(\l \l N')^2 + (\l \l N'')^2]} (N' N'')^{1-\tau}
\een
is indeed always bounded by (\ref{eq:bestphi}).
We can bound (\ref{eq:bestphisubstitution}) by
\ben
K^{-A_1} \delta^{A_2 \l \l N} e^{A_3 (\l \l N)^2} N^{1-\tau} u \cdot v
\een
where
\ben
u &=& (\log N)^{-10 A_1 - 6A_2 \l \l N - 40} e^{\frac{9}{10}A_3 (\l \l N)^2} \\
v &=& \delta^{14A_1 - \log \frac{20}{11} A_2 + 40}.
\een
For suitable choices of $A_2, A_3 > A_1$ both $u, v > 1$ so  (\ref{eq:bestphi}) is admissible. 

Similarly for (\ref{eq:bestpsi}) we have
\ben
(K'K'')^{B_1} (\delta')^{-B_2 \l \l N'} (\delta'')^{-B_2 \l \l N''} e^{-B_3[(\l \l N')^2 + (\l \l N'')^2]} (N' N'')^{\gamma} \\
					<K^{B_1} \delta^{-B_2 \l \l N} e^{-B_3 (\l \l N)^2} N^{\gamma} u \cdot v
\een
with
\ben
u &=& (\log N)^{10 B_1 + 6B_2 \l \l N} e^{-\frac{9}{10}B_3 (\l \l N)^2} \\
v &=& \delta^{-14B_1 + \log \frac{20}{11} B_2 }.
\een
Again, by taking suitable $B_3 > B_2 > B_1$ we make $u, v < 1$ so (\ref{eq:bestpsi}) is admissible. It closes the induction on scales argument and finishes the proof.

\end{proof}

\section{Finishing the proof}
\begin{proof}
Let $\gamma, \tau > 0$ be constants to be defined later. We start with $\mathcal{A}$ and it's $\mathcal{P}$-image $A$. Since $|\mathcal{A} \mathcal{A}| \leq K |\mathcal{A}|$ we have $|A+A| \leq K$. Define $N_1 = N_2 = |A|; \, N := N_1N_2 = |A|^2$ and take $G$ to be the full graph $A \times A$, so $\delta = 1$. By Lemma \ref{lm:finaldamissiblepairs} and the definition of admissible pairs, there is $G' \subset G$ of size
$$
K^{-A_1} e^{A_3 (\log \log N)^2 } N^{1 - \tau}
$$
so that for any vertex $v \in V(G')$ the $\mathcal{P}$-preimage of $N_{G'}(v)$ is 
$$
K^{B_1}  e^{-B_3 (\log \log N)^2 } N^{\gamma}
$$
separating. There is $v \in V(G')$ such that 
$$
|N_{G'}(v)| > K^{-A_1} e^{A_3 (\log \log N)^2 } N^{1/2 - \tau}
$$
which means that there is $\mathcal{A'} \subset \mathcal{A}$ of size at least 
$$
K^{-A_1} |\mathcal{A}|^{1 -  2\tau}
$$
which is 
$$
K^{B_1} e^{-B_3 (\log \log N)^2 } N^{\gamma}
$$
separating. In particular, 
\beq
E_+(\mathcal{A'}) \leq K^{2B_1} e^{-2B_3 (\log \log N)^2 } N^{2\gamma} |\mathcal{A'}|^2 \leq K^{2B_1} |\mathcal{A}|^{2 +4 \gamma }.
\eeq

It remains to show that if a fairly large subset $\mathcal{A'}$ has small energy then $\mathcal{A}$ itself has small energy. We formulate it as a separate combinatorial lemma in a slightly more general setting.
\begin{lemma} \label{lm:energycoverlemma}
Let $A_i, i = 1, \ldots, L$ be a family of sets such that 
$$
A \subset \bigcup_{1\leq i \leq L } A_i,
$$ 
and, moreover, each $a \in A$ is covered by at least $M$ sets $A_i$. Then
$$
E_+(A) \leq \frac{1}{M^4} \left( \sum_{1 \leq i \leq L} E_+(A_i)^{1/4} \right)^4
$$

\end{lemma} 
\begin{proof}	
	Since each $a \in A$ belongs to at least $M$ sets $A_i$, we have
	$$
  M^4 E_+(A) \leq \sum_{1 \leq i, j, k, l \leq L} \sum_{x} 1_{A_i} * 1_{A_j}(x)  1_{A_k} * 1_{A_l}(x).	
	$$	
	Applying the Cauchy-Schwarz inequality twice, we  bound
	\ben
	\sum_{x} 1_{A_i} * 1_{A_j}(x)  1_{A_k} * 1_{A_l}(x) &\leq& \left(\sum_{x}1_{A_i} * 1_{A_j}^2(x) \right)^{1/2} \left(\sum_{x}1_{A_k} * 1_{A_l}^2(x) \right)^{1/2} \nonumber\\
																				  &=& E^{1/2}_+(A_i, A_j)	E^{1/2}_+(A_k, A_l) \nonumber \\
																				  &\leq& E^{1/4}_+(A_i)E^{1/4}_+(A_j)E^{1/4}_+(A_k)E^{1/4}_+(A_l).\nonumber
	\een
	Thus, after summing over the indices $1 \leq i, j, k, l \leq L$ one gets		
	$$
	M^4 E_+(A) \leq \left( \sum_{1\leq i \leq L} E^{1/4}_+(A_i)\right)^4
	$$
	and the claim follows.																  
\end{proof}

It remains to apply Lemma \ref{lm:energycoverlemma}. Take $a \in \mathcal{A}$ and an arbitrary $a' \in \mathcal{A'}$. One can write 
$$
a = \frac{a}{a'}a'.
$$
Thus, taking $A_\alpha := \alpha \mathcal{A'}$  and the covering 
$$
A \subset \bigcup_{\alpha}  A_{\alpha}
$$ with $\alpha \in \frac{\mathcal{A}}{\mathcal{A'}}$, we conclude that each $a \in \mathcal{A}$ is covered by at least $|\mathcal{A'}|$ sets $A_\alpha$. On the other hand, clearly 
$$
E_+(A_\alpha) = E_+(\mathcal{A'}) \leq K^{2B_1} |\mathcal{A}|^{2 + \gamma }.
$$
Also, by the Pl\"unecke-Ruzsa inequality
$$
\left| \frac{\mathcal{A}}{\mathcal{A'}}  \right| \leq \left| \frac{\mathcal{A}}{\mathcal{A}}  \right| \leq K^2 |\mathcal{A}|.
$$
Thus, applying Lemma \ref{lm:energycoverlemma} with $M = |\mathcal{A'}|$ and $L = |\mathcal{A}/\mathcal{A'}|$, we get
$$
E_+(\mathcal{A}) \leq  K^8 \frac{|\mathcal{A}|^4}{|\mathcal{A'}|^4}  K^{2B_1} |\mathcal{A}|^{2 + \gamma } \leq K^{2B_1 + 8 + 4A_1} |\mathcal{A}|^{2 + 8\tau + \gamma}. 
$$
 
By taking $\tau, \gamma$ small enough, we finish the proof of Theorem \ref{thm:maintheorem} since $B_1, A_1$ depend only on $\tau$ and $\gamma$.
\end{proof}

\section{Acknowledgments}
The work on this exposition was supported by a postdoctoral fellowship funded by the Knuth and Alice Wallenberg Foundation. 

The author is indebted to Oliver Roche-Newton and Brandon Hanson for scrupulous proof-reading and valuable discussions which contributed much to the exposition.

\end{document}